\documentclass[11pt]{amsart}
\usepackage{amsmath,amsthm, amscd, relsize, amssymb, amsfonts, mathtools}
\usepackage{amsmath,amsthm, amscd, amssymb, amsfonts,enumerate,graphicx,fancyhdr,nccmath, relsize}
\usepackage{latexsym,xspace,mathtools,color,mathrsfs,caption}
\usepackage{xcolor}
\usepackage{tikz-cd}
\usepackage[all]{xy}
\usepackage{mathrsfs}
\usepackage[inline]{enumitem}
\usepackage[T2A,T1]{fontenc}
\usepackage{bigstrut}
\usepackage{rotating}
\usepackage{hyperref}
\usepackage{multirow}
\usepackage{xcolor}

\usepackage{caption}

\usepackage[ansinew]{inputenc}
\usepackage{graphicx,fancyhdr}

 \hypersetup{colorlinks=true,citecolor=blue,linkcolor=cyan,linktocpage=true}

\newcommand{\comment}[1]{}

\numberwithin{equation}{section}
\theoremstyle{plain}
\newtheorem*{maintheorem}{Theorem}
\newtheorem{theorem}{Theorem}[section]
\newtheorem{lemma}[theorem]{Lemma}
\newtheorem{coro}[theorem]{Corollary}

\newtheorem{prop}[theorem]{Proposition}

\theoremstyle{definition}

\theoremstyle{remark}

\newtheorem{remark}[theorem]{Remark}

\def\pf{\begin{proof}}
\def\epf{\end{proof}}

\newcommand{\mgo}{\mathfrak{m}}
\newcommand{\mfG}{\mathfrak{fsl}(2)}

\newcommand{\odd}{\mathbb{O}}
\newcommand{\even}{\mathbb{E}}

\newcommand{\ku}{ \Bbbk}

\newcommand{\I}{\mathbb I}

\newcommand{\N}{\mathbb N}

\newcommand{\Z}{\mathbb Z}

\newcommand{\Ss}{{\mathcal S}}

\newcommand{\Att}{\mathtt A}
\newcommand{\Btt}{\mathtt B}
\newcommand{\Ctt}{\mathtt C}

\newcommand{\xtt}[1]{\mathtt{#1}}
\newcommand{\xbtt}[1]{\overline{\mathtt{#1}}}

\newcommand{\Utt}{\mathtt U}
\newcommand{\Vtt}{\mathtt V}
\newcommand{\Wtt}{\mathtt W}

\newcommand{\ogr}{\mathtt{a}_1}
\newcommand{\tgr}{\mathtt{a}_2}
\newcommand{\thgr}{\mathtt{a}}
\newcommand{\fgr}{\mathtt b}

\newcommand{\coker}{\operatorname{coker}}

\newcommand{\id}{\operatorname{id}}

\newcommand{\Jac}{\operatorname{J}}

\newcommand{\Hom}{\operatorname{Hom}}

\newcommand{\soc}{\operatorname{soc}}
\newcommand{\rad}{\operatorname{rad}}

\newcommand{\Indec}{\operatorname{indec}\nolimits}
\newcommand\ext{\operatorname{Ext}}

\newcommand{\lmod}[1]{\hspace{-2pt}{}_{#1}\mathcal{M}}
\newcommand{\clmod}[1]{\hspace{-2pt}{}_{#1}\underline{\mathcal{M}}}
\newcommand{\oclmod}[1]{\hspace{-2pt}{}_{#1}\overline{\mathcal{M}}}
\newcommand{\qdim}{\operatorname{qdim}}

\newcommand\rep{\operatorname{rep}}

\newcommand{\Irr}{\operatorname{irrep}}

\allowdisplaybreaks

\newcounter{tabla}\stepcounter{tabla}

\begin{document}

\title[The Green ring of a restricted enveloping algebra]{The Green ring of a restricted enveloping algebra in characteristic 2}

\author[Andruskiewitsch, Bagio, Della Flora, Fl\^ores]
{Nicol\'as Andruskiewitsch, Dirceu Bagio, Saradia Della Flora, Daiana Fl\^ores}

\address[N.~Andruskiewitsch]{CIEM-CONICET. 
	Medina Allende s/n (5000) Ciudad Universitaria, C\'ordoba, Argentina.
	\newline Department of Mathematics and Data
	Science, Vrije Universiteit Brussel, Pleinlaan 2, 1050 Brussels, Belgium}
\email{nicolas.andruskiewitsch@unc.edu.ar}

\address[D.~Bagio]{Departamento de Matem\'atica, Universidade Federal de Santa Catarina,
88040-900, Florianópolis, SC, Brazil}
\email{d.bagio@ufsc.br}

\address[S.~Della Flora]{Departamento de Matem\'atica, Universidade Federal de Santa Maria,
97105-900, Santa Maria, RS, Brazil} \email{saradia.flora@ufsm.br}

\address[D.~Fl\^ores]{Departamento de Matem\'atica, Universidade Federal de Santa Maria, 97105-900, Santa Maria, RS, Brazil} \email{flores@ufsm.br}

\thanks{
N. A.  was partially supported by~CONICET {\small (PIP 11220200102916CO)},
FONCyT-ANPCyT {\small (PICT-2019-03660)} and Secyt (UNC).  D. B. was partially supported by Fundação de Amparo a Pesquisa e Inovação do Estado de Santa Catarina (FAPESC), Edital 21/2024. }

\begin{abstract} Let $\Bbbk$ be an algebraically closed field of 
characteristic $2$ and let $\mathfrak{fsl}(2)$ be the unique, 
up to isomorphism, $3$-dimensional simple Lie algebra over $\Bbbk$. 
Denote by $\mathfrak{m}$ the minimal $2$-envelope 
of $\mathfrak{fsl}(2)$ and by $\mathfrak{u}(\mathfrak{m})$ 
its corresponding restricted enveloping algebra. 
The non-isomorphic finite-dimensional indecomposable
$\mathfrak{u}(\mathfrak{m})$-modules were classified in \cite{ABDF}. 
In this paper, the Green ring (or representation ring) 
for $\mathfrak{u}(\mathfrak{m})$ is calculated. Also, the semisimplification of the representation category of $\mathfrak{u}(\mathfrak{m})$ is determined.
\end{abstract}

\maketitle

\setcounter{tocdepth}{1}
\tableofcontents

\section{Introduction}

Let $\Bbbk$ be an algebraically closed field of characteristic $2$. The fake $\mathfrak{sl}(2)$ is the derived Lie algebra of the $4$-dimensional  Witt Lie algebra, i.e.,  $\mfG \coloneqq W(1,\underline{2})^{'}$, see \cite{cushingetal,str1}. The Lie algebra $\mfG$ has a basis $\{a,b,c\}$ and bracket
\begin{align*}
&[a,b]=c,& &[a,c]=a,& &[b,c]=b.&
\end{align*}
It is known that $\mfG$ is the unique, up to isomorphism, simple Lie algebra of dimension $3$, see \cite[Example 2.4]{cushingetal}. The Lie algebra $\mfG$ is not restricted. If $\mgo$ denotes the minimal $2$-envelope of $\mfG$, then the restricted enveloping algebra of  $\big(\mgo, (\,\,)^{[2]}\big)$ is given by 
$\mathfrak{u}(\mgo) \simeq \ku\langle a,b,c\rangle / I$, where $I$ is the ideal generated by the relations
\begin{align*}
&ab+ba=c,\!& &ac+ca=a,\!& &bc+cb=b,\!& &a^4=b^4=0,\!&&c^2+c=0,&
\end{align*}
see \cite[Remark 2.5]{cushingetal}. Thus,  $\mathfrak{u}(\mgo)$ has a structure of Hopf algebra with $a,b,c$ primitive elements.

The category of finite-dimensional $\mathfrak{u}(\mgo)$-modules was studied by the authors in \cite{ABDF}. Particularly, it was proved that $\mathfrak{u}(\mgo)$ has tame representation type 
and the complete description of non-isomorphic finite-dimensional indecomposable $\mathfrak{u}(\mgo)$-modules was presented. 
Our main purpose in this work is to determine the decomposition in direct sum of indecomposable $\mathfrak{u}(\mgo)$-modules of the tensor product between any indecomposable $\mathfrak{u}(\mgo)$-modules. This information will be encoded in the {\it Green ring} of  $\mathfrak{u}(\mgo)$.

We recall that the Green ring $r(H)$ of a Hopf algebra $H$ consists of the abelian group generated by the isomorphism classes of finite-dimensional $H$-modules with operation $[V\oplus W]=[V]+[W]$, where $[V]$ and $[W]$ denote the isomorphism classes of the finite-dimensional $H$-modules $V$ e $W$, respectively. The multiplication in $r(H)$ is induced by the tensor product between $H$-modules, i.~e., $[V][W]=[V\otimes W]$. The concept
of Green ring was firstly considered for modular representations of finite groups in \cite{G}. From there, numerous works on Green rings have emerged in the literature. Specially in recent years, there has been great interest in this subject. To illustrate, see for instance \cite{chen,KS,LZ,vay} and the references therein.

The finite-dimensional indecomposable $\mathfrak{u}(\mgo)$-modules were classified in \cite[$\S\,3$ and  $\S\,4$]{ABDF} as: simple modules and their projective covers, string modules and band modules. For the purposes of this paper, in  $\S\,3$  we reinterpret such a classification using the same approach as in \cite{chen} and \cite{vay}; that is, we classify the indecomposable modules as syzygy, cosyzygy and $(r,r)$-type modules. The decomposition as direct sum of indecomposable modules of the tensor product between any indecomposable modules is given in $\S \,4$. The main result of this work is Theorem \ref{green-ring} which states the following.
\begin{maintheorem}\label{mainthm:green ring}
Let $\mathbb{Z}[X]$ be the polynomial algebra over $\mathbb{Z}$ in the commutative variables $X=\{x_i, Z_{\xtt{x},s}: i \in \I_{1,4}, \,\xtt{x}\in \mathbb{P}_1(\Bbbk), s \in \mathbb{N} \}$ and $I$ the ideal of $\mathbb{Z}[X]$ generated by the relations \eqref{polynomial} and \eqref{polynomial2}. Then we have the following ring isomorphism
$$r(\mathfrak{u}(\mgo)) \simeq \mathbb{Z}[X]/I.$$
\end{maintheorem}
In the last section, we prove that there exists a monoidal equivalence between the semisimplification $\underline{\rep} \,\mathfrak{u}(\mgo)$ of $\rep\mathfrak{u}(\mgo)$ and the category $\operatorname{vect}_{\Bbbk}^{\Gamma}$  of $\Gamma$-graded finite-dimensional $\Bbbk$-vector spaces, where $\Gamma=C_2\times \Z$.

\section{Preliminaries} 

\subsection{Notations and Conventions} We work over an algebraically closed field $\ku$ of characteristic $2$.
The natural numbers are denoted by $\N$ and $\N_0=\N\cup \{0\}$. Moreover, denote  by $\even=\{2n: n \in \N_0\}$ and $\odd=\{2n+1: n \in \N_0\}$. For $k<\ell \in\N_0$, we set $\I_{k,\ell}=\{n \in \N_0: k \leq n \leq \ell\}$ and $\I_{\ell}=\I_{1, \ell}$. 
 
All vector spaces, algebras and tensor products are over $\ku$. 
If $U$ is a vector space and $S\subseteq U$, then $\ku S$ stands 
for the linear subspace of $U$ generated by $S$.

Given an algebra $A$, let $\lmod{A}$ denote  the category of  finite-dimensional left $A$-modules. 
Also let $\Irr A$, respectively $\Indec A$, denote the set of isomorphism classes of  simple, respectively 
indecomposable, objects  in $\lmod{A}$.
We often denote indistinctly a class in $\Irr A$, or $\Indec A$, and one of its representatives. The Jacobson radical of $A$ is denoted by $\Jac(A)$ and the radical of $M\in \lmod{A}$ is denoted by $\rad M$. Given $M$ in $\lmod{A}$ and $n \in \N$, $nM$ denotes the direct sum of $n$ copies of $M$. Moreover, $\operatorname{P}(M)$ and $\operatorname{I}(M)$ denote the projective cover and the injective envelope of $M$, respectively. 

\subsection{The algebra $\mathfrak{u}(\mgo)$}

Let $\mfG \coloneqq W(1,\underline{2})^{'}$ be the derived Lie algebra of the four-dimensional  Witt Lie algebra,
see \cite{cushingetal,str1}. The Lie algebra $\mfG$ is called the fake $\mathfrak{sl}(2)$ in \cite{cushingetal} and it has a basis $\{a,b,c\}$ and bracket
\begin{align}\label{lie-bracket}
&[a,b]=c,& &[a,c]=a,& &[b,c]=b.&
\end{align}
It is known that $\mfG$ is the unique, up to isomorphism, simple Lie algebra of dimension $3$ \cite[Example 2.4]{cushingetal}. The center of 
$U(\mfG)$ and $\Irr U(\mfG)$ were computed in \cite{do}. 

The Lie algebra $\mfG$ is not restricted. The minimal $2$-envelope of $\mfG$ is the Lie algebra $\mgo$ with basis $\{b',b,c,a,a'\}$, bracket
\eqref{lie-bracket} and
\begin{align*}
&\begin{aligned}
&[a',b]=a,&  &[a',b']=c, & &[a,b']=b,&  
&[a',a]=[a',c]=[b',b]=[b',c]=0;
\end{aligned}
\end{align*}
and $2$-operation $(\,\,)^{[2]}:\mgo\to \mgo$ given by
\begin{align*}
&(a')^{[2]}=(b')^{[2]}=0,& &c^{[2]}=c,& &a^{[2]}=a',& &b^{[2]}=b'.&
\end{align*}
 The restricted enveloping algebra of  $\big(\mgo, (\,\,)^{[2]}\big)$ is given by 
$\mathfrak{u}(\mgo) \simeq \ku\langle a,b,c\rangle / I$, where $I$ is the ideal generated by the relations
\begin{align}\label{rel-u(L)}
&ab+ba=c,\!& &ac+ca=a,\!& &bc+cb=b,\!& &a^4=b^4=0,\!&&c^2+c=0.&
\end{align}
Moreover  $\mathfrak{u}(\mgo)$ has a structure of Hopf algebra with $a,b,c$ primitive elements.

\begin{remark}\label{basis-um}
 It follows from \cite[\S 2]{ABDF} that
$\{a^ib^jc^k\,:\,i,j\in \I_{0,3},\,k\in \I_{0,1}\}$ 
is a basis of $\mathfrak{u}(\mgo)$.
\end{remark}

\subsection{Indecomposable modules of $\mathfrak{u}(\mgo)$} \label{indecomposable} It was proved in \cite[Lemma 3.12]{ABDF} that $\mathfrak{u}(\mgo)$ has tame representation type. In this subsection we will recall the classification of finite-dimensional non-isomorphic indecomposable $\mathfrak{u}(\mgo)$-modules which is explicitly presented in \cite[Theorem 4.1]{ABDF}.

\smallbreak

\begin{theorem}\label{thm:clasif-indec-um} The set of finite-dimensional non-isomorphic indecomposable $\mathfrak{u}(\mgo)$-modules  consists of
\begin{enumerate}[leftmargin=*,label=\rm{(\roman*)}]
\item the simple modules $V_i$ and their projective covers $P_i$, $i\in \I_{0,1}$, which are presented in $\S$ \ref{subsec:simple-proj-cover};

\item the string modules $\Utt_{i,r}$, $\Vtt_{j,t}$, $\Wtt_{j,t}$, $i \in \I_4$, $ j \in \I_{2}$,  $r\in \N$, $ t \in \N_0$,
with action described in \eqref{eq:string-action-basis} and Table \ref{table:string-action};

\item the band modules 
$\Att_{\lambda,r}$,  $\Btt_{\lambda,r}$,  $r \in \N$, $\lambda \in \ku^{\times}$,
with action described in \eqref{eq:band-action-basis} and Table \ref{table:band-action}.
\end{enumerate}
\end{theorem}

\subsubsection{Simple modules and their projective covers}\label{subsec:simple-proj-cover}
The algebra $\mathfrak{u}(\mgo)$ has two finite-dimensional non-isomorphic simple modules, namely, the trivial one-dimensional  $\mathfrak{u}(\mgo)$-module $V_{0}$ and the three-dimensional $\mathfrak{u}(\mgo)$-module $V_1$ with basis $\{v_i: i \in \I_3\}$ and actions of $a,b,c$ described by the matrices 
\begin{equation}\label{simple-module}
\Att=\left(\begin{matrix} 0 & 0 & 0 \\ 1 & 0 & 0 \\ 0 & 1 & 0 \end{matrix}\right), \qquad
\Btt=\left(\begin{matrix} 0 & 1 & 0 \\ 0 & 0 & 1 \\ 0 & 0 & 0 \end{matrix}\right),  \qquad
\Ctt=\left(\begin{matrix} 1 & 0 & 0 \\ 0 & 0 & 0 \\ 0 & 0 & 1 \end{matrix}\right),\qquad
\end{equation}
 respectively. We can represent the action of $\mathfrak{u}(\mgo)$ on $V_1$ through the following directed graph 
 
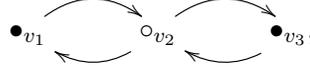
\vspace{0.25cm} 
\begin{align*}
{\xymatrix@C=10mm@R=9mm{   
&\bullet_{v_1} \ar@/^1pc/[r]&\circ_{v_2} \ar@/^1pc/[r] \ar@/^1pc/[l]
&\bullet_{v_3},\ar@/^1pc/[l] }}  \end{align*} 

\vspace{0,25cm} 
\noindent where the arrows oriented with concavity downward, indicate the action of $a$ while the arrows oriented with concavity upward represent the action of $b$. Moreover, $\circ_{u}$ means $cu=0$ while $\bullet_{u}$ means $cu=u$, $u\in V_1$. In this section, all directed graphs representing actions of $\mathfrak{u}(\mgo)$ follow this rule. \vspace{.1cm}

Let $P_i=\operatorname{P}(V_i)$, $i\in \I_{0,1}$. By \cite[Proposition 3.11]{ABDF}, $P_i$ is an 8-dimensional $\mathfrak{u}(\mgo)$-module. 
There is a basis $\{v_1,v_2,v_3,v_4,w_1,w_2,w_3,w_4\}$  of $P_0$
such that the action of $\mathfrak{u}(\mgo)$ on $P_0$ 
is represented by the following directed graph
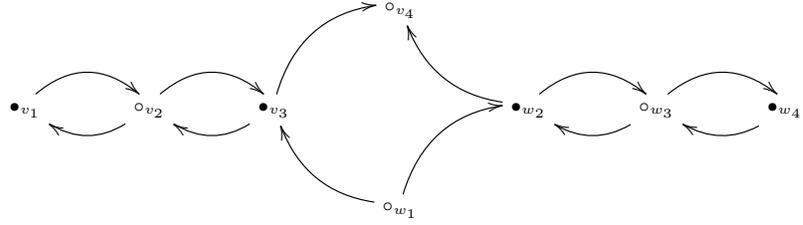
\begin{align*}
{\tiny\xymatrix@C=10mm@R=9mm{   
& & & &\circ_{v_4}&&& \\
&\bullet_{v_1} \ar@/^1pc/[r]
&\circ_{ v_2} \ar@/^1pc/[r] \ar@/^1pc/[l]
&\bullet_{v_3} \ar@/^1pc/[l] \ar@{->}@/^1pc/[ur]
& &\bullet_{w_2}\ar@/^1pc/[r] \ar@{->}@/^1pc/[ul]
&\circ_{w_3}\ar@/^1pc/[r]\ar@/^1pc/[l]
&\bullet_{w_4}.\ar@/^1pc/[l]\\
& & & &\circ_{w_1}\ar@{->}@/^1pc/[ur] \ar@{->}@/^1pc/[ul]&&&}} \end{align*}
\noindent Also, there is a basis $\{v_1,v_2,v_2,v_4,w_1,w_2,w_3,w_4\}$ 
of $P_1$ such that the action of $\mathfrak{u}(\mgo)$ on $P_1$ is represented by the following directed graph
 \vspace{0.25cm} 

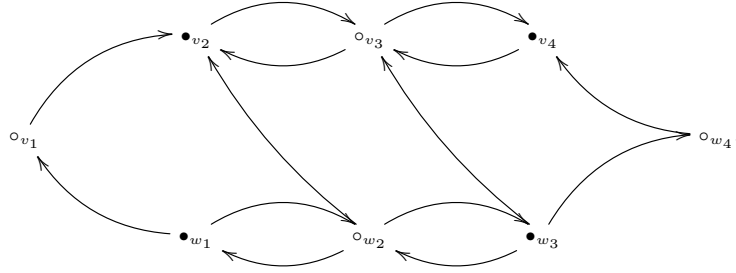
\begin{align*}
{\tiny\xymatrix@C=16mm@R=9mm{
& &\bullet_{v_2} \ar@/^1pc/[r]
&\circ_{v_3} \ar@/^1pc/[r] \ar@/^1pc/[l]
&\bullet_{ v_4} \ar@/^1pc/[l]
& \\
&\circ_{v_1}\ar@{->}@/^1pc/[ur]
&
&
&
&\circ_{w_4}.\ar@{->}@/^1pc/[ul]\\
&&\bullet_{w_1}\ar@/^1pc/[r] \ar@{->}@/^1pc/[ul]
&\circ_{w_2}\ar@/^1pc/[r]\ar@/^1pc/[l]\ar@{->}@/^/[luu]
&\bullet_{w_3}\ar@/^1pc/[l]\ar@{->}@/^/[luu]\ar@{->}@/^1pc/[ur]
&\\
&}} \end{align*}

\vspace{-0,5cm}

The next remarks will be useful later.
\begin{remark} \label{rem-projec-cover-iso}
Denote by $P_{0, \lambda}$, $\lambda\in \Bbbk^\times$, the $\mathfrak{u}(\mgo)$-module  with the same basis of $P_0$ and the same actions of $a,b,c$ on $P_0$ except that $b\cdot w_i=\lambda v_{i+2}$, $i \in \I_2$.

Notice that $ P_{0, \lambda}\simeq P_0$. In fact, take the basis $\{\tilde{v}_1,\tilde{v}_2,\tilde{v}_3,\tilde{v}_4,\tilde{w}_1,\tilde{w}_2,\tilde{w}_3,\tilde{w}_4\}$ of $P_{0,\lambda}$ with  $\tilde{v}_i=\lambda v_i$ and $\tilde{w}_i=w_i$, $i\in \I_4$. The linear map 
$\psi_{0,\lambda}:P_{0, \lambda}\to P_0$ given by 
$\tilde{v}_i\mapsto v_i$  and 
$\tilde{w}_i\mapsto w_i$ is an isomorphism of $\mathfrak{u}(\mgo)$-modules.

Similarly, denote by $P_{1, \lambda}$ the module that has the same basis of $P_1$ and the same actions of $a,b,c$ on $P_1$ except that
\begin{align*}
& b\cdot w_1=\lambda v_1, & & b\cdot w_2=w_1+\lambda v_2, & & b\cdot w_3=w_2+\lambda v_3, & &b\cdot w_4=\lambda v_4.\end{align*} 
Consider the basis $\{\tilde{v}_1,\tilde{v}_2,\tilde{v}_3,\tilde{v}_4,\tilde{w}_1,\tilde{w}_2,\tilde{w}_3,\tilde{w}_4\}$ of $P_{1,\lambda}$  with  $\tilde{v}_i=\lambda v_i$ and $\tilde{w}_i=w_i$, $i\in \I_4$. Hence, the linear map $\psi_{1,\lambda}:P_{1, \lambda}\to P_1$ given by $\tilde{v}_i\mapsto v_i$ and $\tilde{w}_i\mapsto w_i$ is an isomorphism of $\mathfrak{u}(\mgo)$-modules. 
\end{remark}

\begin{remark} \label{dualizing-projec-cover}
Let $H$ be a Hopf algebra with antipode $\Ss$ and $V\in \lmod{H}$. Then, $V^{\ast} \in \lmod{H}$ with action  given by $(h\cdot \varphi)(v)=\varphi(\Ss(h)\cdot v)$, $h\in H$, $\varphi\in V^{\ast}$ and $v\in V$. It is straightforward to verify that $P_i^\ast\simeq P_i$, $i\in \I_{0,1}$.
\end{remark}

\subsubsection{String modules} \label{subsec-string-modules}
 There are eight families of string modules, namely
\begin{align*}
&\Utt_{i,r},&  &
\Vtt_{j,t}, & &\Wtt_{j,t},
\end{align*}
$i \in \I_4$, $j \in \I_{2}$, $r\in \N$, $t \in \N_0$,
with action described in \eqref{eq:string-action-basis} and Table \ref{table:string-action}.
For each string module of dimension $d$, there exists a basis $\{z_i\}_{i \in \I_{d}}$, such that the action is given by
\begin{align}\label{eq:string-action-basis}
a \cdot z_i &= \boldsymbol{\kappa}_{i} z_{i+1}, &
b \cdot z_i &= \boldsymbol{\mu}_{i} z_{i-1}, &
c \cdot z_i &= \boldsymbol{\nu}_{i} z_{i}.
\end{align}
Here $\boldsymbol{\kappa}_{i}$  and $\boldsymbol{\mu}_{i}$ take the values $0$ or $1$ and we specify in Table \ref{table:string-action} the $i$'s where the value is $0$. The Table \ref{table:string-action} also records the value of $\boldsymbol{\nu}_{i}$ 
(which is $0$ or $1$). By convention, $\Utt_{3,0}=V_1$ and $\Utt_{1,0}=V_0$.

\begin{table}[h]
\caption{String modules: coefficients in \eqref{eq:string-action-basis}}
\label{table:string-action}
\begin{tabular}{|c|c|c|c|c|}
\hline
Family & $\dim$  & $\boldsymbol{\kappa}_{i}$  & $\boldsymbol{\mu}_{i}$ & $\boldsymbol{\nu}_{i}$ 
\\ \hline
$\Utt_{1,r}$ & $4r+1$ &  $i \equiv 0 (4)$ or $i = 4r+1$   &  $i \equiv 2 (4)$   or $i = 1$    & $i+1$
\\ \hline
$\Utt_{2,r}$ & $4r+3$ &$i \equiv 3 (4)$    &   $i \equiv 1 (4)$         & $i$
\\ \hline
$\Utt_{3,r}$ & $4r+3$ & $i \equiv 0 (4)$  or $i = 4r+3$    &  $i \equiv 0 (4)$   or $i = 1$         & $i$
\\ \hline
$\Utt_{4,r}$ & $4r+1$ & $i \equiv 1 (4)$    & $i \equiv 1 (4)$           & $i+1$
\\ \hline
$\Vtt_{1,t}$ & $4(t+1)$ & $i \equiv 0 (4)$    &    $i \equiv 2 (4)$  or $i = 1$         & $i+1$
\\ \hline
$\Vtt_{2,t}$ & $4(t+1)$ & $i \equiv 3 (4)$  or $i = 4(t+1)$    &    $i \equiv 1 (4)$        & $i$
\\ \hline
$\Wtt_{1,t}$ & $4(t+1)$ & $i \equiv 0 (4)$   or $i = 4(t+1)$   &   $i \equiv 0 (4)$     or $i = 1$      & $i$
\\ \hline
$\Wtt_{2,t}$ & $4(t+1)$ & $i \equiv 1 (4)$  or $i = 4(t+1)$    &  $i \equiv 1 (4)$  or $i = 1$          & $i+1$
\\  \hline 
\end{tabular}
\end{table}

Below we illustrate the string modules for the case $r=t=2$ \vspace{.5cm}

{\scriptsize
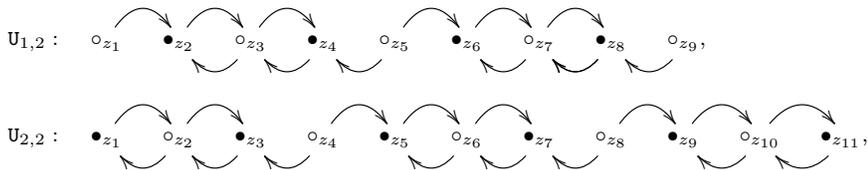
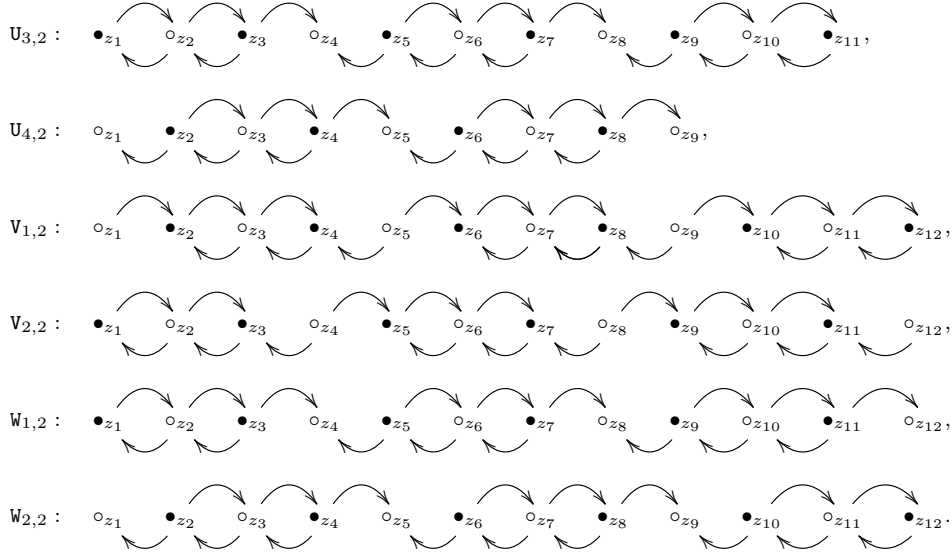
\begin{align*} 
&\Utt_{1,2}:\!\!\!\!\!\!\xymatrix@C=3mm@R=6mm{   
& \circ_{z_1}  \ar@/^1pc/[r]
& \bullet_{z_2} \ar@/^1pc/[r] 
& \circ_{z_3}  \ar@/^1pc/[r] \ar@/^1pc/[l]
& \bullet_{z_4} \ar@/^1pc/[l]
&\circ_{z_5}  \ar@/^1pc/[r]\ar@/^1pc/[l]
&\bullet_{z_6}\ar@/^1pc/[r]
&\circ_{z_7}\ar@/^1pc/[r]\ar@/^1pc/[l]
&\bullet_{z_8}\ar@/^1pc/[l]\ar@/^1pc/[l]
&\circ_{z_9},\ar@/^1pc/[l]
} & \\[.7cm]
&\Utt_{2,2}:\!\!\!\!\!\! \xymatrix@C=3mm@R=6mm{   
& \bullet_{z_1} \ar@/^1pc/[r]
& \circ_{z_2} \ar@/^1pc/[r] \ar@/^1pc/[l]
& \bullet_{z_3}\ar@/^1pc/[l]
& \circ_{z_4}\ar@/^1pc/[r] \ar@/^1pc/[l]
&\bullet_{z_5} \ar@/^1pc/[r]
&\circ_{z_6} \ar@/^1pc/[r] \ar@/^1pc/[l]
&\bullet_{z_7}\ar@/^1pc/[l]
&\circ_{z_8}\ar@/^1pc/[r]\ar@/^1pc/[l]
&\bullet_{z_9}\ar@/^1pc/[r]
&\circ_{z_{10}}\ar@/^1pc/[r]\ar@/^1pc/[l]
&\bullet_{z_{11}},\ar@/^1pc/[l]
}&\\[.7cm]
&\Utt_{3,2}:\!\!\!\!\!\! \xymatrix@C=3mm@R=6mm{   
& \bullet_{z_1} \ar@/^1pc/[r]
& \circ_{z_2} \ar@/^1pc/[r] \ar@/^1pc/[l]
& \bullet_{z_3}\ar@/^1pc/[r] \ar@/^1pc/[l]
& \circ_{z_4} 
&\bullet_{z_5} \ar@/^1pc/[r]\ar@/^1pc/[l]
&\circ_{z_6} \ar@/^1pc/[r]\ar@/^1pc/[l]
&\bullet_{z_7}\ar@/^1pc/[r]\ar@/^1pc/[l]
&\circ_{z_8}
&\bullet_{z_9}\ar@/^1pc/[r]\ar@/^1pc/[l]
&\circ_{z_{10}}\ar@/^1pc/[r]\ar@/^1pc/[l]
&\bullet_{z_{11}},\ar@/^1pc/[l]
}&\\[.7cm]
&\Utt_{4,2}:\!\!\!\!\!\!\xymatrix@C=3mm@R=6mm{   
& \circ_{z_1}  
& \bullet_{z_2} \ar@/^1pc/[r] \ar@/^1pc/[l] 
& \circ_{z_3}  \ar@/^1pc/[r] \ar@/^1pc/[l]
& \bullet_{z_4} \ar@/^1pc/[l] \ar@/^1pc/[r] 
&\circ_{z_5}  
&\bullet_{z_6}\ar@/^1pc/[r] \ar@/^1pc/[l]
&\circ_{z_7}\ar@/^1pc/[r]\ar@/^1pc/[l]
&\bullet_{z_8}\ar@/^1pc/[r]\ar@/^1pc/[l]
&\circ_{z_9},
} & \\[.7cm]
&\Vtt_{1,2}:\!\!\!\!\!\! \xymatrix@C=3mm@R=6mm{   
&\circ_{z_1} \ar@/^1pc/[r]
&\bullet_{z_2} \ar@/^1pc/[r]  
&\circ_{z_3} \ar@/^1pc/[r] \ar@/^1pc/[l]
&\bullet_{z_4} \ar@/^1pc/[l]
&\circ_{z_5} \ar@/^1pc/[r]\ar@/^1pc/[l]
&\bullet_{z_6}\ar@/^1pc/[r]
&\circ_{z_7}\ar@/^1pc/[r]\ar@/^1pc/[l]
&\bullet_{z_8}\ar@/^1pc/[l]\ar@/^1pc/[l]
&\circ_{z_9}\ar@/^1pc/[r]\ar@/^1pc/[l]
&\bullet_{z_{10}}\ar@/^1pc/[r]
&\circ_{z_{11}}\ar@/^1pc/[r]\ar@/^1pc/[l]
&\bullet_{z_{12}},\ar@/^1pc/[l]
}&\\[.7cm]
&\Vtt_{2,2}:\!\!\!\!\!\! \xymatrix@C=3mm@R=6mm{   
&\bullet_{z_1} \ar@/^1pc/[r]
&\circ_{z_2} \ar@/^1pc/[r] \ar@/^1pc/[l]
&\bullet_{z_3} \ar@/^1pc/[l]
&\circ_{z_4} \ar@/^1pc/[r]\ar@/^1pc/[l]
&\bullet_{z_5} \ar@/^1pc/[r]
&\circ_{z_6}\ar@/^1pc/[r]\ar@/^1pc/[l]
&\bullet_{z_7}\ar@/^1pc/[l]
&\circ_{z_8}\ar@/^1pc/[r]\ar@/^1pc/[l]
&\bullet_{z_9}\ar@/^1pc/[r]
&\circ_{z_{10}}\ar@/^1pc/[r]\ar@/^1pc/[l]
&\bullet_{z_{11}}\ar@/^1pc/[l]
&\circ_{z_{12}},\ar@/^1pc/[l]
}&\\[.7cm]
&\Wtt_{1,2}:\!\!\!\!\!\! \xymatrix@C=3mm@R=6mm{   
&\bullet_{z_1} \ar@/^1pc/[r]
&\circ_{z_2} \ar@/^1pc/[r] \ar@/^1pc/[l]
&\bullet_{z_3} \ar@/^1pc/[r] \ar@/^1pc/[l]
&\circ_{z_4} 
&\bullet_{z_5} \ar@/^1pc/[r]\ar@/^1pc/[l]
&\circ_{z_6}\ar@/^1pc/[r]\ar@/^1pc/[l]
&\bullet_{z_7}\ar@/^1pc/[r]\ar@/^1pc/[l]
&\circ_{z_8}
&\bullet_{z_9}\ar@/^1pc/[r]\ar@/^1pc/[l]
&\circ_{z_{10}}\ar@/^1pc/[r]\ar@/^1pc/[l]
&\bullet_{z_{11}}\ar@/^1pc/[r]\ar@/^1pc/[l]
&\circ_{z_{12}},
}&\\[.7cm]
&\Wtt_{2,2}:\!\!\!\!\!\! \xymatrix@C=3mm@R=6mm{   
&\circ_{z_1} 
&\bullet_{z_2}\ar@/^1pc/[l] \ar@/^1pc/[r] 
&\circ_{z_3} \ar@/^1pc/[r] \ar@/^1pc/[l]
&\bullet_{z_4} \ar@/^1pc/[r] \ar@/^1pc/[l]
&\circ_{z_5}
&\bullet_{z_6}\ar@/^1pc/[r]\ar@/^1pc/[l]
&\circ_{z_7}\ar@/^1pc/[r]\ar@/^1pc/[l]
&\bullet_{z_8}\ar@/^1pc/[r]\ar@/^1pc/[l]
&\circ_{z_9}
&\bullet_{z_{10}}\ar@/^1pc/[r]\ar@/^1pc/[l]
&\circ_{z_{11}}\ar@/^1pc/[r]\ar@/^1pc/[l]
&\bullet_{z_{12}}.\ar@/^1pc/[l]
}&
\end{align*}} \vspace{.2cm}

\subsubsection{Band modules}\label{subsec-band-modules}
There are two families of  band modules, denoted by $\Att_{\lambda,r}$,  $\Btt_{\lambda,r}$,  $r\in \N$, $\lambda \in \ku^{\times}$,
with action described in \eqref{eq:band-action-basis} and 
Table \ref{table:band-action}, which have dimension $4r$. 
Any of $\Att_{\lambda,r}$ or $\Btt_{\lambda,r}$ has
a basis  $\{z_i: i \in \I_{4r}\}$  
such that  the action of $\mathfrak{u}(\mgo)$ is given by 
\begin{align}\label{eq:band-action-basis}
a \cdot z_i &= \boldsymbol{\kappa}_{i} z_{i+1}, &
b \cdot z_i &= \boldsymbol{\mu}_{i} z_{i-1} + \boldsymbol{\xi}_i \lambda z_{i+3}, &
c \cdot z_i &= \boldsymbol{\nu}_{i} z_{i}.
\end{align}
Here $\boldsymbol{\kappa}_{i}$,  $\boldsymbol{\mu}_{i}$  and $\boldsymbol{\xi}_{i}$ take the values $0$ or $1$; 
we specify in Table \ref{table:band-action} the $i$'s where the value is $0$. The Table \ref{table:band-action} also records the value of $\boldsymbol{\nu}_{i}$ 
(which is $0$ or $1$).

\begin{table}[ht]
\caption{Band modules: coefficients in \eqref{eq:band-action-basis}}
\label{table:band-action}
\begin{tabular}{|c|c|c|c|c|}
\hline
Family &  $\boldsymbol{\kappa}_{i}$  & $\boldsymbol{\mu}_{i}$ & $\boldsymbol{\xi}_{i}$ & $\boldsymbol{\nu}_{i}$ 
\\ \hline
$\Att_{\lambda,r}$   &  $i \equiv 0 (4)$    &  $i \equiv 2 (4)$ or $i=1$  & $i \equiv 0, 2, 3 (4)$  & $i+1$
\\ \hline
$\Btt_{\lambda,r}$   &$i \equiv 0 (4)$    &   $i \equiv 0 (4)$ or $i=1$ &  $i \equiv 0, 2, 3 (4)$     & $i$
\\ \hline
\end{tabular} 
\end{table}

We illustrate the band modules above (for $r=2$) via a directed graph.
Here we have that $b\cdot z_1=\lambda z_{4}$ and $b \cdot z_5=z_4+\lambda z_{8}$; to describe this, 
we label the arrows at the bottom of this diagram with $\lambda$.
\vspace{.3cm}

{\scriptsize
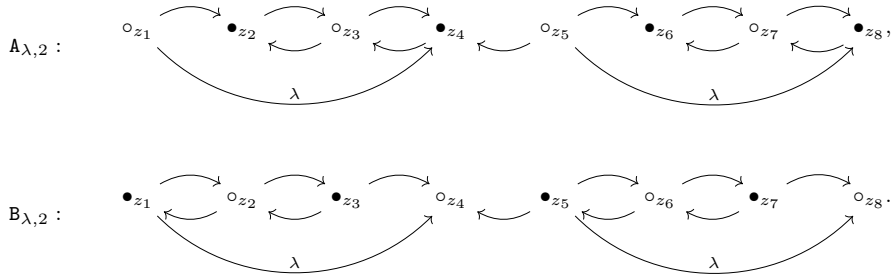
\[\Att_{\lambda,2}:\qquad 
\begin{tikzcd}
\circ_{z_1} \arrow[r,bend left=30] \arrow[rrr,bend right=45,"\lambda"] &
\bullet_{z_2}\arrow[r,bend left=30] &
\circ_{z_3} \arrow[r,bend left=30] \arrow[l,bend left=30]&
\bullet_{z_4}\arrow[l,bend left=30]& 
\circ_{z_5}\arrow[r,bend left=30]\arrow[l,bend left=30] \arrow[rrr,bend right=45,"\lambda"]&
\bullet_{z_6}\arrow[r,bend left=30]&
\circ_{z_7}\arrow[r,bend left=30]\arrow[l,bend left=30] &
\bullet_{z_8},\arrow[l,bend left=30]\\
\end{tikzcd}
\]
\vspace{.3cm}
\[\Btt_{\lambda,2}:\qquad 
\begin{tikzcd}
\bullet_{z_1} \arrow[r,bend left=30] \arrow[rrr,bend right=45,"\lambda"] &
\circ_{z_2}\arrow[r,bend left=30]\arrow[l,bend left=30] &
\bullet_{z_3} \arrow[r,bend left=30] \arrow[l,bend left=30]&
\circ_{z_4}& 
\bullet_{z_5}\arrow[r,bend left=30]\arrow[l,bend left=30] \arrow[rrr,bend right=45,"\lambda"]&
\circ_{z_6}\arrow[r,bend left=30] \arrow[l,bend left=30]&
\bullet_{z_7}\arrow[r,bend left=30]\arrow[l,bend left=30] &
\circ_{z_8}.\\
\end{tikzcd}
\]}

\subsection{Some properties of the indecomposable modules}
We start by recalling some definitions. Let $A$ be a finite-dimensional algebra, $J=\Jac (A)$ and $M\in \lmod{A}$. Assume that $\dim M<\infty$.  The smallest non-negative integer $i$ satisfying $J^{i} M=0$ is called the {\it radical length} (or {\it Loewy length}) of  $M$ and denoted by $\operatorname{rl}(M)$. The series $0 \subseteq J^{i-1} M \subseteq \dots \subseteq J^{2} M  \subseteq J M \subseteq M$ is called the {\it radical series} of $M$.  

For $j > 1$, we define recursively $\soc^ j M$ in the following way. 
Being   $\soc(M/\soc^{j-1} M)$  a submodule of $M/\soc^{j-1} M$, 
there exists a unique submodule $\soc^ j M$ of $M$ 
containing $\soc^{j-1} M$
 such that  $\soc^ j M/ \soc^ {j-1} M$ is isomorphic to 
 $\soc(M/\soc^{j-1} M)$. 
 The smallest integer $t\geq 1$ with $\soc^{t}M=M$ 
 is called the {\it socle length} of $M$, 
 denoted by $\operatorname{sl}(M)$, 
 and $0 \subseteq \soc M \subseteq \soc^{2} M  \subseteq \dots \subseteq \soc^{t-1}M \subseteq M$
is the {\it socle series} of $M$. 
It is well known that $\operatorname{rl}{M}= \operatorname{sl}(M)$; see, for instance, \cite[Proposition II 4.7]{ars}.

\begin{prop}\label{length3} Consider $\mathfrak{u}(\mgo)$ as a module over itself via
left multiplication. Then $\operatorname{rl}(\mathfrak{u}(\mgo))=3$;
in particular, $\Jac^3(\mathfrak{u}(\mgo))=0$.
\end{prop}
\begin{proof}
Notice that $\soc P_0 =  \ku\{v_4\} \simeq V_0$. Moreover, $$\soc (P_0/\soc P_0) =\ku\{\overline{v_1}, \overline{v_2}, \overline{v_3}, \overline{w_2}, \overline{w_3}, \overline{w}_4\} \simeq 2 V_1$$ as $\mathfrak{u}(\mgo)$-module. Consequently, $\soc^2 P_0 =\{v_1,v_2,v_3,v_4,w_2,w_3,w_4\}$ and $P_0/\soc^2 P_0=\ku\{w_1\}\simeq V_0$. Therefore, $\soc^3P_0 \simeq P_0$ and $\operatorname{rl}{(P_0)}= \operatorname{sl}(P_0)=3$. Similarly $\soc^3 P_1 \simeq P_1$ and $\operatorname{rl}{(P_1)}= \operatorname{sl}(P_1)=3$. Since $\mathfrak{u}(\mgo) \simeq P_0 \oplus 3P_1$ as $\mathfrak{u}(\mgo)$-modules, then $\operatorname{rl}(\mathfrak{u}(\mgo))=3$.
\end{proof}

\begin{prop}
Let $M \in \Indec \mathfrak{u}(\mgo)$. The following assertions hold.
\begin{enumerate}[leftmargin=*,label=\rm{(\roman*)}]
\item If $\operatorname{rl}(M)=1$, then $M \simeq V_0$ or $M \simeq V_1$. 
\item If $\operatorname{rl}(M)=3$, then $M \simeq P_0$
or $M \simeq P_1$.
\end{enumerate}
\end{prop}
\begin{proof}
The item (i) is evident. By \cite[Corollary 8.4.3]{Ra}, $\mathfrak{u}(\mgo)$ is a Frobenius algebra; hence  it is self-injective. By Proposition \ref{length3}, $\operatorname{rl}(\mathfrak{u}(\mgo))= 3 = \operatorname{rl}(M)$. Follows from \cite[Lema 3.5]{chen1} that $M$ is projective, consequently $M \simeq P_0$
or $M \simeq P_1$.
\end{proof}

\begin{remark} \label{remark-socle-rad}
Let $A$ be a finite-dimensional algebra and $M \in \Indec A$ such that $\operatorname{rl}(M)=2$. By \cite[Lemma 3.7]{chen1}, $\rad M=\soc M$. In particular, for any $ M \in \Indec \mathfrak{u}(\mgo)$ in Tables \ref{table:string-action} and \ref{table:band-action}, $\rad M=\soc M$.
\end{remark}
Let $A$ be an algebra and $M \in \lmod{A}$. In the next, we use the notation $\overline{M}=M/\rad M$.

\begin{prop} \label{prop-rad-families} Let $r  \in \N$, $\lambda \in \ku^{\times}$, we have the following:  
\begin{align*}
&\rad\Utt_{1,r}\simeq \rad\Att_{\lambda,r}\simeq rV_1,&&\rad\Utt_{2,r}\simeq \rad\Vtt_{1,r}\simeq\rad\Vtt_{2,r}\simeq (r+1)V_1,&\\[.2em]
&\rad\Utt_{3,r}\simeq \rad\Btt_{\lambda,r}\simeq rV_0,& &\rad\Utt_{4,r}\simeq \rad\Wtt_{1,r}\simeq \rad\Wtt_{2,r}\simeq  (r+1)V_0,& \\[.2em]
& \rad \Vtt_{1,0} \simeq \rad \Vtt_{2,0} \simeq V_1, & & \rad \Wtt_{1,0} \simeq \Wtt_{2,0} \simeq V_0.
\end{align*}
Moreover,
\begin{align*}
&\overline{\Utt}_{1,r}\simeq \overline{\Vtt}_{1,r}\simeq  \overline{\Vtt}_{2,r}\simeq (r+1)V_0,&& \overline{\Utt}_{2,r}\simeq \overline{\Att}_{\lambda,r}\simeq rV_0, \\
&\overline{\Utt}_{3,r}\simeq \overline{\Wtt}_{1,r}\simeq  \overline{\Wtt}_{2,r}\simeq(r+1)V_1,&& \overline{\Utt}_{4,r}\simeq  \overline{\Btt}_{\lambda,r}\simeq rV_1,\\
& \overline{\Vtt}_{1,0} \simeq \overline{\Vtt}_{2,0} \simeq V_0, &
& \overline{\Wtt}_{1,0} \simeq \overline{\Wtt}_{2,0} \simeq V_1.
\end{align*}
\end{prop}
\begin{proof}
It is clear that $\rad\Utt_{1,r}=\Bbbk \{ z_i : i\in\I_{4r}, i \not\equiv 1 \pmod 4 \}\simeq rV_1$. Hence $\overline{\Utt}_{1,r}=\Bbbk\{\overline{z}_{4i+1}: i\in \I_{0,r}\}\simeq (r+1)V_0$. The proofs for the other families are similar.
\end{proof}

For any $r \in \N$, $t \in \N_0$, $i \in \I_2$, $j\in \I_{3,4}$, using the Proposition \ref{prop-rad-families}, we have the following exact sequences of modules
\begin{align}
\label{ex_Uir}& 0 \to (r+i-1)V_1 \to \Utt_{i,r} \to (r-i+2)V_0 \to 0,\\[.2em]
\label{ex_Ujr}& 0 \to (r+j-3)V_0 \to \Utt_{j,r} \to (r-j+4)V_1 \to 0,\\[.2em]
\label{ex_Vit}&0 \to (t+1)V_1 \to \Vtt_{i,t} \to (t+1)V_0 \to 0,\\[.2em]
\label{ex_Wit}&0 \to (t+1)V_0 \to \Wtt_{i,t} \to (t+1)V_1 \to 0, \\[.2em]
\label{ex_A}& 0 \to rV_1 \to \Att_{\lambda,r} \to rV_0 \to 0, \\[.2em]
\label{ex_B} &0 \to rV_0 \to \Btt_{\lambda,r} \to rV_1 \to 0.
\end{align} 

\section{Another parametrization of indecomposable modules}

In order to investigate tensor products in $\Indec\mathfrak{u}(\mgo)$, we present in this section another description of the finite-dimensional non-isomorphic indecomposable $\mathfrak{u}(\mgo)$-modules. We follow the same approach as in \cite{chen1}.

\subsection{Syzygy and cosyzygy modules}
Let $A$ be a finite-dimensional algebra and $M,N\in \lmod{A}$.  Consider the vector subspace $\mathcal{P}(M,N)$ of $\operatorname{Hom}(M,N)$ consisting of the morphisms $M\to N$ which factor through a projective $A$-module, that is, $f\in \mathcal{P}(M,N)$ if there exist a projective module $P$ and module morphisms $g:M\to P$ and $h:P\to N$ such that $hg=f$.
Denote by $\clmod{A}$ the factor category $\lmod{A}/\mathcal{P}$. The objects in $\clmod{A}$ are the same of $\lmod{A}$ and the space of morphisms from $M$ to $N$ in $\clmod{A}$ is the quotient space $\operatorname{Hom}(M,N)/\mathcal{P}(M,N)$. The category $\clmod{A}$  is usually called the {\it stable module category} of $A$, see \cite[p. 37]{carlson}.

Now, we recall the definition of the {\it syzygy functor} $\Omega:\clmod{A}\to \clmod{A}$. For each $M\in \lmod{A}$, choose a fixed projective cover $P(M)\overset{f}\to M\to 0$ and define $\Omega(M)=\ker f$; see \cite[ch. IV]{ars} for details. We consider inductively $\Omega^i:\clmod{A}\to \clmod{A}$ given by
\begin{align*}
\Omega^0=\id_{\clmod{A}},& & \Omega^{i+1}=\Omega\Omega^i, \,\,\,i\in \N_0.
\end{align*}

On the other hand, let $M,N\in \lmod{A}$ and $\mathcal{I}(M,N)$ the vector subspace of $\operatorname{Hom}(M,N)$ consisting of the morphisms  $M\to N$ which factor through an injective $A$-module. We denote  by $\oclmod{A}$ the factor category  $\lmod{A}/\mathcal{I}$.  The {\it cosyzygy functor} $\Omega^{-1}:\oclmod{A}\to \oclmod{A}$ is defined as follows. For each  $M\in \lmod{A}$, choose a fixed injective envelope $0\to I(M)\overset{f}\to M$ and define $\Omega^{-1}(M)=\coker f$. We define $\Omega^{-i}:\oclmod{A}\to \oclmod{A}$ inductively by 
\begin{align*}
\Omega^{-(i+1)}=\Omega^{-1}\Omega^{-i}, \,\,\,i\in \N.
\end{align*}

Suppose that $A$ is self-injective. Since the projective and injective $A$-modules coincide, we have that $\mathcal{P}(M,N)=\mathcal{I}(M,N)$, for all $M,N\in \lmod{A}$. Thus, $\clmod{A}=\oclmod{A}$. Moreover, it follows from \cite[Proposition 3.5]{ars} that $\Omega$ and $\Omega^{-1}$ are inverse equivalences.

We recall that any finite-dimensional Hopf algebra is Frobenius, see for instance \cite[Corollary 8.4.3]{Ra}. Thus,  any finite-dimensional Hopf algebra is self-injective.

\begin{theorem}\label{teo-syzygy}
Let $r\in \N_0$, $t\in \N$, $l,k \in \I_{0,1}$, $l \neq k$. Then,
$$\Omega^r(V_l)\simeq \begin{cases}
\Utt_{3-2l,r}, & \text{if } r \in \odd,\\
\Utt_{1+2l,r}, & \text{if } r \in \even,\\
\end{cases}, \qquad \Omega^{-t}(V_l) \simeq \begin{cases}
\Utt_{2+2l,t}, & \text{if } t \in \odd,\\
\Utt_{4-2l,t}, & \text{if } t \in \even.\\
\end{cases} $$
Moreover, we have the following exact sequences of $\mathfrak{u}(\mgo)$-modules
\begin{align}
\label{sec-omega1-r-odd}& 0 \to \Omega^{r+1}(V_l) \rightarrow (r+1)P_k\rightarrow \Omega^r(V_l) \to 0,\,\, r \in \odd,\\[.2em] 
\label{sec-omega1-r-even}& 0 \to \Omega^{r+1}(V_l) \rightarrow (r+1)P_l \rightarrow \Omega^r(V_l) \to 0,\,\,  r \in  \even, \\[.2em] 
\label{sec-coomega1-r-odd}& 0 \to \Omega^{-t}(V_l) \rightarrow (t+1)P_k \rightarrow \Omega^{-(t+1)}(V_l) \to 0,\,\, t \in \odd,\\[.2em] 
\label{sec-coomega1-r-even}& 0 \to \Omega^{-t}(V_l) \rightarrow (t+1)P_l \rightarrow \Omega^{_(t+1)}(V_l) \to 0,\,\, t \in \even.
\end{align}
\end{theorem}

\begin{proof}
 We prove the result for syzygies modules in the case where $l=0$; the proof for $l=1$ is similar. We proceed by induction on $r$.
The case $r=0$ follows immediately from the convention $\Omega^0(V_0)=V_0=\Utt_{1,0}$. Take $r=1$. Let $\{v_1,v_2,v_3,v_4,w_1,w_2,w_3,w_4\}$ be a basis of $P_0$ and $\{u\}$ be a basis of $V_0$. Define the module epimorphism $\phi:P_0\to V_0$ by
\begin{align*}
& \phi(v_i)=0, & & \phi(w_1)=u, & &\phi(w_j)=0,\,\,\,i\in \I_4,\, j\in \I_{2,4}.
\end{align*}
It is clear that $\Omega(V_0)=\ker \phi\simeq \Utt_{3,1}$. 

Suppose that $r>1$ and that the result is valid for $r-1$. We have two cases. \vspace{.2cm}

\noindent{\it Case 1.} $r$ is odd.

Since $r-1$ is even, it follows by the inductive hypothesis that $\Omega^{r-1}(V_0)=\Utt_{1,r-1}$. By Proposition \ref{prop-rad-families}, $\overline{\Utt}_{1,r-1}\simeq rV_0$.

Observe that the map  $\Phi=\phi\oplus\cdots\oplus\phi:rP_0\to rV_0$ is an epimorphism. Consider $\pi:\Utt_{1,r-1}\to \overline{\Utt}_{1,r-1}$ be the natural projection. Since $rP_0$ is a projective module, there is a morphism $q: rP_0\to \Utt_{1,r-1}$ such that $\pi\circ q=\Phi$ and we obtain that $(rP_0,q)$ is the projective cover of $\Utt_{1,r-1}$. In order to present $q$ explicitly, for each $i\in \I_r$, we consider a module $P_i\simeq P_0$ with basis $\beta_i=\{v_{i1},v_{i2},v_{i3},v_{i4},w_{i1},w_{i2},w_{i3},w_{i4}\}$, where the isomorphism is determined by  $v_{ij}\mapsto v_j$ and $w_{ij}\mapsto w_j$, $j\in \I_4$. Then, $\beta:=\cup_{i \in \I_r}\beta_i$ is a basis to $rP_0$ and $q$ is given by
\begin{align*}
&q(v_{ij})=\begin{cases}
0, & \text{if }  (i,j)\in \{1\}\times \I_{4} \text{ or } (i,j)\in \I_{2,r}\times \{4\},\\
z_{4(i-2)+(j+1)}, & \text{if } (i,j)\in \I_{2,r}\times \I_{1,3},\\
\end{cases}&\\[.4em]
&
q(w_{ij})=\begin{cases}
z_{4(i-1)+j}, & \text{if } (i,j)\in \I_{r-1} \times \I_4,\\
z_{4(r-1)+1}, & \text{if } (i,j)=(r,1),\\
0, & \text{if } (i,j)\in \{r\} \times \I_{2,4}.\\
\end{cases}&
\end{align*}
Consider the set $Y_i:=\{v_{i4},w_{i2}+v_{(i+1)1},w_{i3}+v_{(i+1)2},w_{i4}+v_{(i+1)3}\}\subseteq rP_0$, $i\in \I_{r-1}$ and $Y=\cup_{i \in \I_{r-1}} Y_i$. Then 
\[\gamma=\{v_{11},v_{12},v_{13}\}\cup Y \cup \{v_{r4}, w_{r2},w_{r3},w_{r4}\}\]
is a basis to $\ker q=\Omega^r(V_0) \simeq \Utt_{3,r}$.
Moreover, we obtain that the sequence given in \eqref{sec-omega1-r-odd} is exact for $j=1$. \vspace{.2cm}

\noindent{\it Case 2.} $r$ is even.

Since $r-1$ is odd, it follows by inductive hypothesis that $\Omega^{r-1}(V_0)=\Utt_{3,r-1}$. Again, by Proposition \ref{prop-rad-families}, $\overline{\Utt}_{3,r-1} \simeq r V_1$.
 
Let $\{v_1,v_2,v_3,v_4,w_1,w_2,w_3,w_4\}$ be a basis of $P_1$ and $\{u_1,u_2,u_3\}$ be a basis of $V_1$.
We define the epimorphism $\psi:P_1\to V_1$
\begin{align*} &\psi(v_i)=0,  && \psi(w_j)=u_j, && \psi(w_4)=0, \,\,\, \,i\in \I_4, j\in \I_{3}.\end{align*} 
Observe that $\Psi=\psi\oplus\cdots\oplus\psi:rP_1\to rV_1$ is an epimorphism.  Also, let $\pi:\Utt_{3,r-1}\to \overline{\Utt}_{3,r-1}$ be the natural projection. Since $rP_1$ is a projective module, there exists a morphism $p: rP_1\to \Utt_{3,r-1}$ such that $\pi\circ p=\Phi$ and we have that $(rP_1,p)$ is the projective cover of $\Utt_{3,r-1}$. Consider for each $i\in \I_r$, a module $P_i\simeq P_1$ with basis $\beta_i=\{v_{i1},v_{i2},v_{i3},v_{i4},w_{i1},w_{i2},w_{i3},w_{i4}\}$, where the isomorphism is determined by the association $v_{ij}\mapsto v_j$ and $w_{ij}\mapsto w_j$, for $j\in \I_4$. Then $\beta:= \cup_{i \in \I_r}\beta_i$ is a basis to $rP_1$ and $p$ is given by
\begin{align*}
&p(v_{ij})=\begin{cases}
0, & \text{if } (i,j)=(1,1) \text{ or } (i,j)\in \I_{r} \times \I_{2,4},\\
z_{4(i-1)}, & \text{if } j=1,\\
\end{cases}&\\[.4em]
&
p(w_{ij})=\begin{cases}
0, & \text{if } (i,j)=(r,4),\\
z_{4i}, & \text{if } (i,j)\in \I_{r-1} \times \{4\},\\
z_{4(i-1)+j}, & \text{if } j\neq 4.\\
\end{cases}&
\end{align*}
For each $i\in \I_{r-1}$, consider the set $X_i:=\{v_{i2},v_{i3},v_{i4},w_{i4}+v_{(i+1)1}\}\subseteq rP_1$ and $X=\cup_{i \in \I_{r-1}}X_i$. Then,
\[\gamma=\{v_{11}\}\cup X \cup \{v_{r2}, v_{r3},v_{r4},w_{r4}\}\]
is a basis to $\ker p=\Omega^r(V_0) \simeq \Utt_{1,r}$. Moreover, we obtain that the sequence given in \eqref{sec-omega1-r-even} is exact for $i=1$.

In a similar way, we can prove the result for cosyzygies. \end{proof}

\begin{remark}\label{dualizing_syzygies}
From \cite[Remark 4.8]{ABDF} we have  that
\begin{align}\label{eq:dual}
\Utt_{1,r}^{\ast} &\simeq \Utt_{4,r}, & 
\Utt_{2,r}^{\ast} &\simeq   \Utt_{3,r}, &
\Vtt_{1,t} ^{\ast} &\simeq   \Wtt_{1,t},   &
\Vtt_{2,t} ^{\ast} &\simeq   \Wtt_{2,t},   &
\Att_{\lambda,r}^{\ast}  &\simeq   \Btt_{\lambda,r},
\end{align}
for all $r\in \N$, $t\in \N_0$, $\lambda \in \ku^{\times}$. Consequently, it follows from the previous result that
\[\Omega^{-s}(V_i)\simeq \big(\Omega^{s}(V_i)\big)^{\ast}, \quad s\in \N,\,\,i\in\I_{0,1}. \] 
\end{remark}

Let $r \in \N$, $l,k \in \I_{0,1}$, $l \neq k$. It is follows from \eqref{ex_Uir} and \eqref{ex_Ujr} that 
\begin{align}\label{ex_Omegar_1}
&0 \to rV_l \rightarrow \Omega^r(V_l)\rightarrow (r+1)V_k \to 0,\,\, r \in \odd;
\end{align}

\begin{align}\label{ex_Omegar_2}
&0 \to rV_l \rightarrow \Omega^r(V_k)\rightarrow (r+1)V_k \to 0,\,\, r \in \even.
\end{align}

\subsection{$(r,r)$-type modules}\label{subsection-type}
Let $A$ be a finite-dimensional algebra and $M \in  \Indec A $, with $\operatorname{rl}(M)=2$. We recall from \cite{chen} that $M$ is called  an $(m,n)$-type module if $\operatorname{l}(\soc M)=n$ and  $\operatorname{l}(M/\soc M)=m$, where $l(U)$ denotes the length of any $A$-module $U$.

\begin{remark}\label{remark-k-ktype} Let $r \in \N$ and $\lambda\in \ku^\times$.  It follows from Theorem \ref{thm:clasif-indec-um}, Remark \ref{remark-socle-rad} and Proposition \ref{prop-rad-families} that the $(r,r)$-type indecomposable $\mathfrak{u}(\mgo)$-modules are  $\Vtt_{1,r-1}$, $\Vtt_{2,r-1}$, $\Wtt_{1,r-1}$, $\Wtt_{2,r-1}$, $\Att_{\lambda,r}$ and $\Btt_{\lambda,r}$. 
\end{remark}

Let us now see that all $(r,r)$-type indecomposable $\mathfrak{u}(\mgo)$-modules can be encompassed by a single family and its dual family. In fact,  
given a pair $\xtt{0}\neq\xtt{x}=(x_1,x_2)\in \ku^2$ with $x_1\neq x_2$ and $r\in \N$, consider the $4r$-dimensional $\mathfrak{u}(\mgo)$-module $\Att_{\xtt{x}}(r)$ with basis $\{z_i: i \in \I_{4r}\}$ and action defined by
\begin{align*}
 a\cdot z_i&=\begin{cases}
x_1 z_{i+1}, & \text{if } i=1,\\
x_1 z_{i+1}+x_2z_{i-3}, & \text{if } i=4t+1,\,t \in \N,\\
0, & \text{if } i=4t,\,t\in \N,\\
z_{i+1}, & \text{otherwise},\\
\end{cases} &  &\\
 b\cdot z_i&=\begin{cases}
x_2 z_{i+3}, & \text{if } i=1,\\
x_1 z_{i-1}+x_2 z_{i+3}, & \text{if } i=4t+1,\,t\in \N,\\
0, & \text{if } i=4t-2,\,t \in \N,\\
z_{i-1}, & \text{otherwise},\\
\end{cases}\\
 c\cdot z_i&=\begin{cases}
z_{i}, & \text{if } i \in \even,\\
0, & \text{otherwise}.\\
\end{cases}
\end{align*}
It is straightforward to check that this action indeed defines an $\mathfrak{u}(\mgo)$-module structure on $\Att_{\xtt{x}}(r)$. Moreover, we will see below that $\Att_{\xtt{x}}(r)$ is indecomposable. As in $\S\,$\ref{subsec-band-modules}, we illustrate the previous module (for $r=2$ and $x_1\neq x_2$) via the following directed graph:

{\scriptsize
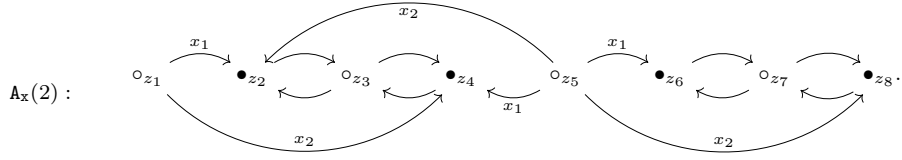
\[\Att_{\xtt{x}}(2):\qquad 
\begin{tikzcd}
\circ_{z_1} \arrow[r,"x_1", bend left=30] \arrow[rrr,bend right=45,"x_2"] &
\bullet_{z_2}\arrow[r,bend left=30] &
\circ_{z_3} \arrow[r,bend left=30] \arrow[l,bend left=30]&
\bullet_{z_4}\arrow[l,bend left=30]& 
\circ_{z_5}\arrow[r,"x_1",bend left=30]\arrow[l, "x_1", bend left=30] \arrow[rrr,bend right=45,"x_2"]  \arrow[lll, "x_2", bend left=-45]&
\bullet_{z_6}\arrow[r,bend left=30]&
\circ_{z_7}\arrow[r,bend left=30]\arrow[l,bend left=30] &
\bullet_{z_8}.\arrow[l,bend left=30]\\
\end{tikzcd}\]}
\begin{remark}\label{remark-decomp}
Note that if we allow in the previous definition that $\xtt{0}\neq\xtt{x}=(x,x)\in \ku^2$, then the corresponding module $\Att_{\xtt{x}}(r)$ is decomposable. In fact, consider the subspace  $S_1$  of $\Att_{\xtt{x}}(r)$ generated by $\{z_i: i \in \I_4\}$. For $i\in \I_{2,r}$, let  $S_i$  the subspace of $\Att_{\xtt{x}}(r)$ generated by $ \{z_1+z_5+\cdots+z_{4i-3},z_{4i-2},z_{4i-1},z_{4i}\}$. Clearly, $S_i$ is a submodule of $\Att_{\xtt{x}}(r)$ and $\Att_{\xtt{x}}(r)=S_1\oplus\cdots\oplus S_r$.
\end{remark}

However for our purposes we establish the following convention   \[\Att_{\xtt{x}}(r):=\Att_{1,r}, \quad \xtt{x}=(x,x)\in \ku^2,\quad x\in \ku^\times.\]

\begin{lemma}\label{lem:type-modules}
Let $r\in \N$ and $x_1,x_2\in \ku^\times$. The following assertions hold.
\begin{enumerate}[leftmargin=*,label=\rm{(\roman*)}]
\item If $\xtt{x}=(x_1,0)$, then  $\Att_{\xtt{x}}(r)\simeq \Vtt_{1,r-1}$. \vspace{.1cm}
\item  If $\xtt{x}=(0,x_2)$,  then  $\Att_{\xtt{x}}(r)\simeq  \Vtt_{2,r-1}$. \vspace{.1cm}
\item  If $\xtt{x}=(x_1,x_2)$ and $x_1\neq x_2$, then 
$\Att_{\xtt{x}}(r)\simeq  \Att_{x_2/x_1,r}$.
\end{enumerate}
\end{lemma}
\pf 
It is straightforward to check (i) and (ii). For the proof of (iii), let $\xtt{x}=(x_1,x_2)\in \ku^2$, $x_1\neq x_2$. Let  $\{z_i\,:\,i\in \I_{4r}\}$ be the basis of $\Att_{\xtt{x}}(r)$. Note that $\Att_{(x_1,x_2)}(r)\simeq \Att_{(1,x_2/x_1)}(r)$ taking the basis $\{w_i: i \in \I_{4r}\}$, where

\begin{align*}
w_{4j+1} &=\begin{cases}
z_{4j+1}, & \text{if } j\in \I_{0,r-1},\\
x_1z_{i},  & \text{ otherwise}. 
\end{cases}
\end{align*} 

Now, taking the basis $\{\tilde{z}_{i}: i \in \I_{4r}\}$ of $\Att_{(1,\gamma)}(r)$, $\gamma=x_2/x_1$, defined by 
\begin{align*}
&\tilde{z}_{4j+i}=(1+\gamma^2)\sum_{k=0}^{j}\gamma^{j-k}w_{4k+i},&  &j\in\I_{0,r-2},\,i\in\I_4,&\\
&\tilde{z}_{4r-3}=\sum_{k=0}^{r-1}\gamma^{r-k-1}w_{4k+1},& &\tilde{z}_{4r-i}=w_{4r-i},\,i\in \I_{0,2},&
\end{align*}
we obtain $\Att_{(1,\gamma)}(r)\simeq \Att_{\gamma,r}$. 
\epf

Let $r\in \N$ and $\xtt{0}\neq \xtt{x}=(x_1,x_2)\in \ku^2$. Consider \[\Btt_{\xtt{x}}(r):=\Att^*_{\xtt{x}}(r)\] be the dual of $\Att_{\xtt{x}}(r)$. The explicit description of  $\Btt_{\xtt{x}}(r)$, for the case $x_1\neq x_2$, is given by:
\begin{align*}
a\cdot z_i&=\begin{cases}
x_1 z_{i+1}, & \text{if } i=3,\\
x_1 z_{i+1}+x_2z_{i-3}, & \text{if } i=4t+3,\,t\in \N,\\
0, & \text{if } i=4t,\,t\in \N,\\
z_{i+1}, & \text{otherwise},\\
\end{cases} &  &\\
b\cdot z_i&=\begin{cases}
x_2 z_{i+3}, & \text{if } i=1,\\
x_1 z_{i-1}+x_2 z_{i+3}, & \text{if } i=4t+1,\, t\in \N,\\
0, & \text{if } i=4t,\, t\in \N,\\
z_{i-1}, & \text{otherwise},\\
\end{cases}\\
c\cdot z_i&=\begin{cases}
z_{i}, & \text{if } i\text{ is odd},\\
0, & \text{otherwise}.\\
\end{cases}
\end{align*}
For $\xtt{0}\neq \xtt{x}=(x,x)\in \ku^2$ we have by \eqref{eq:dual} that $\Btt_{\xtt{x}}(r)=\Btt_{1,r}$.
The next directed graph illustrates $\Btt_{\xtt{x}}(2)$ (when $\xtt{x}=(x_1,x_2)$ and $x_1\neq x_2$):
 
 {\scriptsize
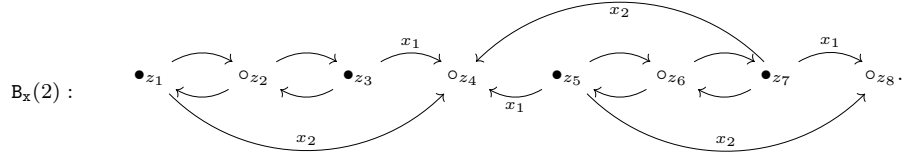
 \[\Btt_{\xtt{x}}(2):\qquad 
 \begin{tikzcd}
 \bullet_{z_1} \arrow[r, bend left=30] \arrow[rrr,bend right=45,"x_2"] &
 \circ_{z_2}\arrow[r,bend left=30] \arrow[l,bend left=30]&
 \bullet_{z_3} \arrow[r,bend left=30,"x_1"] \arrow[l,bend left=30]&
 \circ_{z_4}& 
 \bullet_{z_5}\arrow[r,bend left=30]\arrow[l, "x_1", bend left=30] \arrow[rrr,bend right=45,"x_2"] &
 \circ_{z_6}\arrow[r,bend left=30] \arrow[l,bend left=30]&
 \bullet_{z_7}\arrow[r,bend left=30,"x_1"]\arrow[l,bend left=30]  \arrow[lll, "x_2", bend left=-45] &
 \circ_{z_8}.\\
 \end{tikzcd}
 \]}

 \begin{lemma}\label{lem:type-modules-dual}
 Let $r \in \N$ and $x_1,x_2\in \ku^\times$. The following assertions hold.
 \begin{enumerate}[leftmargin=*,label=\rm{(\roman*)}]
 \item If $\xtt{x}=(x_1,0)$,  then  $\Btt_{\xtt{x}}(r)\simeq \Wtt_{1,r-1}$. \vspace{.1cm}
 \item  If $\xtt{x}=(0,x_2)$, then  $\Btt_{\xtt{x}}(r)\simeq  \Wtt_{2,r-1}$. \vspace{.1cm}
 \item  If $\xtt{x}=(x_1,x_2)$ and $x_1\neq x_2$, then 
 $\Btt_{\xtt{x}}(r)\simeq  \Btt_{x_2/x_1,r}$.
 \end{enumerate}
 \end{lemma}
 \pf Since $\Btt_{\xtt{x}}(r)$ is the dual of $\Att_{\xtt{x}}(r)$, the result follows directly from \eqref{eq:dual} and Lemma \ref{lem:type-modules} .
 \epf 
 
 Let $\mathbb{P}_1(\Bbbk)$ the one-dimensional projective space over $\Bbbk$, i.e., the elements in  $\mathbb{P}_1(\Bbbk)$ are the classes of the following equivalence relation on $\ku^2$:
 \[(a,b)\sim (c,d)\text{ if and only if there exists $\lambda \in \ku^\times$ such that } (a,b)=\lambda(c, d).\]
 The equivalence class of an element $\xtt{x}=(x_1,x_2)\in \ku^2$ will be denoted by $\xbtt{x}$.\vspace{.1cm}
 
 \begin{prop}\label{properties-types}
 Let $r \in \N$. The $(r,r)$-type indecomposable modules with socle $rV_i$ are parametrized by  $\mathbb{P}_1(\Bbbk)$, for each $i\in \I_{0,1}$.
 \end{prop}
\pf 
Let $V$ be an $(r,r)$-type indecomposable module and assume that $\soc V=rV_1$. From Proposition \ref{prop-rad-families}
and Remark \ref{remark-k-ktype} follow that $V\simeq V_{1,r-1}$,  $V\simeq V_{2,r-1}$ or $V\simeq \Att_{\lambda,r}$, for some $\lambda\in \ku^\times$. By Lemma \ref{lem:type-modules} we have that
\[V_{1,r-1}\simeq \Att_{\xtt{x}_1}(r), \quad V_{2,r-1}\simeq \Att_{\xtt{x}_2}(r), \quad \Att_{\lambda,r}\simeq \Att_{\xtt{x}_3}(r),  \]
where $\xtt{x}_1\in \overline{(1,0)}$, $\xtt{x}_2\in \overline{(0,1)}$ and $\xtt{x}_3\in \overline{(1,\lambda)}$. Finally, by convention, $\Att_{1,r}\simeq \Att_{\xtt{x}_4}(r)$ with $\xtt{x}_4=(x,x)$ and $x\neq 0$. The proof for the case $\soc V=rV_0$ is similar.
\epf
 
  Let $r \in \N$ and $0\neq \xtt{x}\in \ku^2$. It follows from \eqref{ex_Vit}, \eqref{ex_Wit}, \eqref{ex_A} and \eqref{ex_B} that 
 \begin{align}
 \label{1seq-Ax} &0 \to rV_1 \to \Att_{\xtt{x}}(r) \to rV_0 \to 0, \\[.2em]
 \label{2seq-Bx}  &0 \to rV_0 \to \Btt_{\xtt{x}}(r) \to rV_1 \to 0.
 \end{align}

In the next proposition we obtain exact sequences analogous to those presented in Theorem \ref{teo-syzygy} for the case of  $(r,r,)$-type modules. These are fundamental for the description of tensor products that will be presented in the following section.

\begin{prop} \label{exact-sequence-Ax-Bx}
Let $r \in \N$. We have the following exact sequences of $\mathfrak{u}(\mgo)$-modules:
\begin{align}
\label{seq1-Ax} &0 \to \Att_{\xtt{x}}(r) \rightarrow rP_1  \rightarrow \Btt_{\xtt{x}}(r)\to 0, \\
\label{seq1-Bx} &0 \to \Btt_{\xtt{x}}(r) \rightarrow rP_0  \rightarrow \Att_{\xtt{x}}(r)\to 0.
\end{align}
\end{prop}
\pf
We prove \eqref{seq1-Ax} for the case $\xtt{x}=(x_1,x_2)\in \ku^2$ with $x_1,x_2 \in \ku^{\times}$, $x_1\neq x_2$. Consider $\lambda=x_2/x_1$, by  Lemmas \ref{lem:type-modules} and \ref{lem:type-modules-dual}, $\Att_{\xtt{x}}(r)\simeq \Att_{\lambda,r}$ and $\Btt_{\xtt{x}}(r)\simeq \Btt_{\lambda,r}$. Now let $\{v_1,v_2,v_3,v_4,w_1,w_2,w_3,w_4\}$ be the basis of $P_1$ given in $\S\ref{subsec:simple-proj-cover}$. From Remark \ref{rem-projec-cover-iso} we have that $ P_{1, \lambda}\simeq P_1$. If $r=1$, then the submodule $\ku\{v_1+w_4, v_2,v_3,v_4\}$ of $ P_{1, \lambda}$ is isomorphic to $\Att_{\lambda,1}$ and $ P_{1, \lambda}/\Att_{\lambda,1} \simeq \Btt_{\lambda,1}$. For $r \geq 2$, consider the following vectors in $rP_{1,\lambda,}$:
\begin{align*}
u_{1,1}&=(v_1+w_4,0,\ldots,0),\\
u_{l,1}&=(0,\ldots,0,w_{4},v_1+w_4,0,\ldots,0),& &l\in \I_{2,r},\\
u_{l,k}&=(0,\ldots,0,v_{k},0,\ldots,0),& &l\in \I_{r},\,\,k\in \I_{2,4},
\end{align*}
where $w_{4}$ is in the $(l-1)$-th position of the vectors $u_{l,1}$ and $v_{k}$ is in the $l$-th position of the vectors $u_{l,k}$. Given $l\in \I_r$, consider $\beta_l=\{u_{l,k}\,:\,k\in \I_4\}$ and $\beta(r)= \cup_{l \in \I_{r}}\beta_l$. It is straightforward to check that the submodule $\tilde{\Att}_{\lambda,r}$ of $rP_{1, \lambda}$ with basis $\beta(r)$ is isomorphic to $\Att_{\lambda,r}$ and $rP_{1, \lambda}/\tilde{\Att}_{\lambda,r}\simeq \Btt_{\lambda,r}$. Using the isomorphism $ P_{1, \lambda}\simeq P_1$ we obtain the result. In a similar way we can prove that \eqref{seq1-Ax} is valid for the other cases, that is, when $\xtt{x}\sim (1,0)$ or $\xtt{x}\sim (0,1)$ or $\xtt{x}\sim (1,1)$. The proof of \eqref{seq1-Bx} is similar.\epf
 
 \begin{lemma}\label{aux1}
Let $r \in \N$  and $\xtt{0}\neq \xtt{x}\in \ku^2$. We have the following isomorphism: 
\begin{enumerate}[leftmargin=*,label=\rm{(\roman*)}]
\item $\Hom_{\mathfrak{u}(\mgo)}(P_0, \Att_{\xtt{x}}(1)) \simeq \ku$;
\item $\Hom_{\mathfrak{u}(\mgo)}(\Att_{\xtt{x}}(r), \Att_{\xtt{x}}(1)) \simeq \ku$;
\item $\Hom_{\mathfrak{u}(\mgo)}(\Btt_{\xtt{x}}(r), \Att_{\xtt{x}}(1)) \simeq \ku^r$. 
\end{enumerate}
\end{lemma}
\pf
(i) Let $\{v_1,v_2,v_3,v_4,w_1,w_2,w_3,w_4\}$ be a basis of $P_0$ as in \S \ref{subsec:simple-proj-cover} and $\{z_i: i \in \I_4\}$ be a basis of $ \Att_{\xtt{x}}(1)$. For each $\lambda \in \ku$, define $\psi_{\lambda}: P_0 \to  \Att_{\xtt{x}}(1)$ by $\psi_{\lambda}(v_i)= \lambda x_2 z_{i+1}$, $\psi_{\lambda}(v_4)=0$, $\psi_{\lambda}(w_1)=\lambda z_1$ and  $\psi_{\lambda}(w_i)= \lambda x_1 z_i$, $i \in \I_3$. The map $\psi: \ku \to \Hom_{\mathfrak{u}(\mgo)}(P_0, \Att_{\xtt{x}}(1))$ given by $\psi(\lambda)=\psi_{\lambda}$ defines a vector spaces isomorphism.
 
(ii) Let $\{z_i: i \in  \I_{4r}\}$ be a basis of $\Att_{\xtt{x}}(r)$ and $\{z_i^{\prime}: i \in  \I_{4}\}$ be a basis of $\Att_{\xtt{x}}(1)$. In this case, defining $\psi_{\lambda}: \Att_{\xtt{x}}(r) \to  \Att_{\xtt{x}}(1)$  by $\psi_{\lambda}(z_i)= 0$, $\psi_{\lambda}(z_j)= \lambda z_j^{\prime}$, $i \in \I_{4(r-1)}$, $j \in \I_{4r-3,4r}$ we obtain the result.

(iii) Let $\{z_i: i \in  \I_{4r}\}$ be a basis of $\Btt_{\xtt{x}}(r)$ and $\{z_i^{\prime}: i \in  \I_{4}\}$ be a basis of $\Att_{\xtt{x}}(1)$. In this case, defining $\psi_{\lambda_1, \dots, \lambda_r}: \Btt_{\xtt{x}}(r) \to  \Att_{\xtt{x}}(1)$  by 
\begin{align*}
& \psi_{\lambda_1, \dots, \lambda_r}(z_1)= \lambda_1x_2z_4^{\prime}, & & \psi_{\lambda_1, \dots, \lambda_r}(z_2)= \psi_{\lambda_1, \dots, \lambda_r}(z_3)=0, \\& \psi_{\lambda_1, \dots, \lambda_r}(z_i)= 0, & & \psi_{\lambda_1, \dots, \lambda_r}(z_j)= \lambda_{\frac{j-1}{4}+1} z_2^{\prime}, \\& \psi_{\lambda_1, \dots, \lambda_r}(z_l)= \lambda_{\frac{l-2}{4}+1}  z_3^{\prime}, & & \psi_{\lambda_1, \dots, \lambda_r}(z_k)= \lambda_{\frac{k-3}{4}+1}  z_4^{\prime}, 
\end{align*}
where $i \equiv 0 (4)$,  $j \equiv 1 (4)$,  $l \equiv 2 (4)$ and  $k \equiv 3 (4)$,
 we obtain the isomorphism.
\epf

 \begin{prop} \label{exact-sequence-rrtypes}
 Let $r\in \N$ and $\xtt{0}\neq \xtt{x}\in \ku^2$. We have the following exact sequences of $\mathfrak{u}(\mgo)$-modules:
 \begin{align}
 \label{seq-Ax} &0 \to \Att_{\xtt{x}}(1) \rightarrow \Att_{\xtt{x}}(r+1)  \rightarrow \Att_{\xtt{x}}(r)\to 0, \\
 \label{seq-Bx} &0 \to \Btt_{\xtt{x}}(1) \rightarrow \Btt_{\xtt{x}}(r+1)  \rightarrow \Btt_{\xtt{x}}(r)\to 0.
 \end{align}
Moreover, any module fitting in \eqref{seq-Ax} (resp. \eqref{seq-Bx}) is isomorphic to either or $N\simeq \Att_{\xtt{x}}(r+1)$ (resp. $N\simeq \Btt_{\xtt{x}}(r+1)$) or $N\simeq \Att_{\xtt{x}}(1)\oplus \Att_{\xtt{x}}(r)$ (resp. $N\simeq \Btt_{\xtt{x}}(1)\oplus \Btt_{\xtt{x}}(r))$.
 \end{prop}
 
 \pf Assume that $\{z'_j: j \in \I_4\}$ and $\{z_j\,:\,j\in \I_{4(r+1)}\}$ are basis of $\Att_{\xtt{x}}(1)$ and $\Att_{\xtt{x}}(r+1)$, respectively. Consider the map $\iota: \Att_{\xtt{x}}(1)\to \Att_{\xtt{x}}(r)$ defined by $\iota(z'_j)=z_j$,  $j\in \I_4$. Clearly $\iota$ is an injective $\mathfrak{u}(\mgo)$-module morphism and $\Att_{\xtt{x}}(r+1)/\iota(\Att_{\xtt{x}}(1))\simeq \Att_{\xtt{x}}(r)$. Thus, the exactness of \eqref{seq-Ax} is proved.  
 Now the exact sequence \eqref{seq1-Bx} induces the following long exact sequence:
 
  \begin{align}
  \begin{aligned}\label{hom} 
  0 & \to \Hom_{\mathfrak{u}(\mgo)}(\Att_{\xtt{x}}(r), \Att_{\xtt{x}}(1)) \rightarrow \Hom_{\mathfrak{u}(\mgo)}(rP_0, \Att_{\xtt{x}}(1))\\
    &\rightarrow \Hom_{\mathfrak{u}(\mgo)}(\Btt_{\xtt{x}}(r), \Att_{\xtt{x}}(1)) \rightarrow \ext^{1}_{\mathfrak{u}(\mgo)}(\Att_{\xtt{x}}(r),\Att_{\xtt{x}}(1)) \to 0.
    \end{aligned} 
  \end{align} 
 
 By Lemma \ref{aux1}, 
 \begin{align*} &\Hom_{\mathfrak{u}(\mgo)}(\Att_{\xtt{x}}(r), \Att_{\xtt{x}}(1)) \simeq \Hom_{\mathfrak{u}(\mgo)}(P_0, \Att_{\xtt{x}}(1)) \simeq \ku,\\ & \Hom_{\mathfrak{u}(\mgo)}(\Btt_{\xtt{x}}(r), \Att_{\xtt{x}}(1)) \simeq \ku^r. 
 \end{align*}
 
 Therefore, according to the sum of the dimensions of the modules in the sequence \eqref{hom},  $\dim \ext^{1}_{\mathfrak{u}(\mgo)}(\Att_{\xtt{x}}(r),\Att_{\xtt{x}}(1))=1$. The proof of \eqref{seq-Bx} is similar. \epf

\section{Tensor product between indecomposable modules}\label{sub:fusion rules}

In this section, we decompose the tensor product of every pair of indecomposable modules into the direct sum of indecomposable modules, inspired in the ideas of \cite{chen}. 
We start by recalling some well-know facts that will be useful in the sequel of the paper.

\subsection{Some background}
From now on we will use without mentioning the following well-known results:

\begin{enumerate}
\item [$\diamond$] for each $\mathfrak{u}(\mgo)$-module $L$, the functors $-\otimes_{\Bbbk} L$ and $L \otimes_{\Bbbk}-$  are exact; 
\item [$\diamond$] for each $\mathfrak{u}(\mgo)$-module $L$, the functor that associates its dual $L^{\ast}=\operatorname{Hom}_{\Bbbk}(L,\Bbbk)$ is exact;
\item [$\diamond$] given a projective $\mathfrak{u}(\mgo)$-module $P$ and a $\mathfrak{u}(\mgo)$-module $L$, the tensor product $P\otimes_{\Bbbk} L\simeq L\otimes_{\Bbbk} P$ is a projective  $\mathfrak{u}(\mgo)$-module.
\end{enumerate}

The next result is probably well known. For the sake of completeness, we include a proof that was suggested by C. Vay.

\begin{lemma}
Let $A$ be an algebra and  

\[\xymatrix{
0 \ar[r] &M\oplus I\ar[r]^>>>>>f&L\ar[r]^>>>>>g &N\oplus P\ar[r]&0 }\]
an exact sequence of left $A$-modules with $I$ injective and $P$ projective. Then, there exists a left $A$-module $\widetilde{L}$ such that $L \simeq \widetilde{L} \oplus I \oplus P$ and  $\widetilde{L}$ fits in the exact sequence 
\[\xymatrix{
0 \ar[r] &M\ar[r]&\tilde{L}\ar[r] &N\ar[r]&0\, .}\]
 \end{lemma}

\pf  Consider the epimorphism $\tilde{g}:L\to P$ given by $\tilde{g}=\pi_P\circ g$, where $\pi_P$ is the canonical projection from $N\oplus P$ to $P$. Since $P$ is projective, there exists a submodule $L_1$ of $L$ such that $L\simeq L_1\oplus P$ and $M\oplus I\simeq f(M\oplus I)\subset L_1$. Hence, the map $\tilde{f}=f\circ \iota_I:I\to L_1$ is a monomorphism, where $\iota_I$ the natural inclusion of $I$ on $M\oplus I$. Using that $I$ is injective, there exists a submodule $\tilde{L}$ of $L_1$ such that $L_1\simeq \tilde{L}\oplus I$ and $M\simeq f(M)\subset \tilde{L}$. Thus $L\simeq \tilde{L}\oplus I\oplus P$. Hence we have an exact sequence 
\[\xymatrix{
0 \ar[r] &M\oplus I\ar[r]&\tilde{L}\oplus I\oplus P\ar[r] &N\oplus P\ar[r]&0 }\]
wich implies that $N\oplus P\simeq \big(\tilde{L}\oplus I\oplus P\big)/\big(M\oplus I\big)\simeq \big( \tilde{L}/M\big)\oplus P$. Consequently
$ \tilde{L}/M\simeq N$ and the result is proved.
\epf 

\begin{remark} \label{obs-self-injective} Notice that by \cite[Cor. 8.4.3]{Ra} $\mathfrak{u}(\mgo)$ is a Frobenius algebra. In this case, we can apply the previous lemma when $I$ and $P$ are projective. This will be used many times throughout this work.
\end{remark}

The next result is inspired in \cite[Lemma 3.12]{chen} and it play an important role in the description of the tensor product into direct sum of indecomposable modules. 


\begin{lemma}\label{lemma-chen-geral}
Let $A$ be a self-injective artinian algebra, $S_j$ the simple $A$-modules, $Q_j=P(S_j)$ finitely ge\-ne\-ra\-ted and $M$ a finitely ge\-ne\-ra\-ted $A$-module such that $\overline{M}=M/ \rad M  \simeq k_j S_j$, for some $k_j \in \N$,  $j\in \I_n$. If
$f:\mathop{\mathlarger{\mathlarger{\mathlarger{\oplus}}}}_{j\in \I_n} s_jQ_j \to M$ 
is a module epimorphism with $s_i \in \N$ and $ s_j \in \N_0$, $j\neq i$, then $s_i \geq k_i$ and $\ker f \simeq \Omega(M) \mathop{\mathlarger{\mathlarger{\oplus}}}\,(s_i-k_i)Q_i \mathop{\mathlarger{\mathlarger{\oplus}}}_{j\neq i} s_jQ_j$.
\end{lemma}
\pf Set $N_j=s_jQ_j$,  $j\in \I_n$. Since $\overline{M}\simeq k_iS_i$, a projective cover of $M$ is an epimorphism $k_iQ_i\twoheadrightarrow M$, and consequently  $s_i\geq k_i$.
Observe that the module epimorphism $f$ induces a module epimorphism $\overline{f}: \mathop{\mathlarger{\mathlarger{\mathlarger{\oplus}}}}_{j\in \I_n} \overline{N_j} \to \overline{M}$.
Using that $ \overline{N_j}\simeq s_jS_j\,\text{ and }\, \overline{M}\simeq k_iS_i$,  we obtain $\overline{f}(\overline{N_i})= \overline{M}$ and $f(N_j)\subseteq \rad M$, $j\neq i$. Note that $N_i=L\oplus T$, where $L=k_iQ_i$ and $T=(s_i-k_i)Q_i$. Moreover, $\overline{L}\simeq k_iS_i$ implies that 
$\overline{f}(\overline{L})= \overline{M}$. Therefore, $f|_{L}:L\to M$ is an epimophism such that $\overline{f}_{\overline{L}}: \overline{L}\to \overline{M}$ is an isomorphism.  By \cite[Lemme 2.1]{As}, $f|_L:L\to M$ is the projective cover of 
$M$. Also, by \cite[Lemme 1.7]{As}, we have that $f|_{N_i}:N_i\to M$ is an epimorphism. As $N_i$ is projective and $f|_L$ is the projective cover of $M$, there exists an epimorphism $\varphi:N_i\to L$ such that 
$f|_L\circ\varphi=f|_{N_i}$. Thus, $\ker \varphi=T$ and $\ker f|_{N_i}\simeq \Omega(M)\oplus T=\Omega(M)\oplus (s_i-k_i)Q_i$. 

Observe that if $s_j=0$, $j\neq i$, we are done. Otherwise, we consider $U= \mathop{\mathlarger{\mathlarger{\mathlarger{\oplus}}}}_{j\neq i} s_jQ_j\neq 0$ and $f_1=f|_U:U\to M$. Since $U$ is projective and $f|_{N_i}$ is an epimorphism, there exist $\psi:U\to N_i$ such that $f|_{N_i}\circ\psi=f_1$. Now, we define $\theta: U\to N_i\oplus U$ by $\theta(x)=\psi(x)-x$, for all $x\in U$. It is clear that $\theta$ is a module monomorphism and we have that $U\simeq \theta(U)=U'$ and $U'$ is a submodule of $N_i\oplus U$. Since $A$ is self-injective, we have that $S_i=\soc S_i=\soc I(S_i)=\soc Q_i$. Consequently, 
$\soc U'\cap \soc N_i\simeq \big(\mathop{\mathlarger{\mathlarger{\mathlarger{\oplus}}}}_{j\neq i} s_jS_j\big) \cap s_iS_i=0$, and $U'\cap N_i=0$. Thus, $N_i\oplus U=N_i\oplus U'$. Using that $U'\subseteq \ker f$ we obtain that 
$\ker f=\ker f|_{N_i}\oplus U'\simeq \Omega(M)\mathlarger{\mathlarger{\oplus}} (s_i-k_i)Q_i \mathop{\mathlarger{\mathlarger{\oplus}}}_{j\neq i} s_jQ_j,$
and the result is proved.
\epf

\begin{coro} \label{cor-epimorphism} Let $M$ be a finite-dimensional $\mathfrak{u}(\mgo)$-module and assume that $\overline{M}=M/ \rad M  \simeq k_i V_i$, for some $k_i \in \N$, $i \in \I_{0,1}$ . If $f:s_0P_0\oplus s_1P_1 \to M$ is a module epimorphism, where $s_i \in \N$ and $ s_j \in \N_0$, then $s_i \geq k_i$ and $\ker f \simeq \Omega(M)\mathop{\mathlarger{\mathlarger{\oplus}}} (s_i-k_i)P_i  \mathop{\mathlarger{\mathlarger{\oplus}}}_{j\neq i}s_jP_j$.
\end{coro}
\pf  By \cite[Corollary 8.4.3]{Ra}, $\mathfrak{u}(\mgo)$ is a Frobenius algebra and hence it is self-injective. Also, $\mathfrak{u}(\mgo)$ is artinian because $\dim \mathfrak{u}(\mgo)<\infty$. Thus, the result follows directly from Lemma \ref{lemma-chen-geral}. \epf

\subsection{Tensoring simple by simple} 
Since $V_0$ is the trivial module, it is clear that $V_0\otimes M\simeq M$, for any $\mathfrak{u}(\mgo)$-module $M$. Let $\{v_i: i \in \I_3\}$ be the basis of $V_1$ given in \S \ref{subsec:simple-proj-cover} and $v_{ij}=v_i\otimes v_j$, $i,j \in \I_3$. Consider the basis of $V_1\otimes V_1$:  $$\mathcal{B}=\{v_{11},v_{12},v_{13},v_{21},v_{22},v_{23},v_{31},v_{32},v_{33}\}.$$ 
It is easy to check that the actions of $a,b,c \in  \mathfrak{u}(\mgo)$ on $V_1\otimes V_1$ in the basis $\mathcal{B}$ are described, respectively, by the matrices
\begin{align*}
& \left( \begin{matrix} \Att & 0 & 0 \\ \id & \Att& 0 \\ 0 & \id & \Att \end{matrix} \right),&  
&\left( \begin{matrix} \Btt & \id & 0 \\ 0 & \Btt & \id \\ 0 & 0 & \Btt \end{matrix} \right),& 
&\quad  \left( \begin{matrix} \Ctt+\id & 0 & 0 \\ 0 &  \Ctt & 0 \\ 0 & 0 &  \Ctt+\id \end{matrix} \right),  &
\end{align*}
where $\Att, \Btt, \Ctt$ are given in \eqref{simple-module},
$\id$ is the $3\times 3$ identity matrix.

\begin{prop}\label{simple-simple}
$V_1\otimes V_1\simeq V_0\oplus P_1$.
\end{prop} 
\pf
Consider $\mathcal{C}=\{u_i: i \in \I_8\}$, where
\begin{align*}
&u_1 = v_{11},& &u_2 =  v_{12}+v_{21},& &u_3 =  v_{13}+v_{31},& &u_4 = v_{23}+v_{32},& \\
&u_5 =  v_{12},& &u_6 = v_{13}+v_{22},& &u_7 = v_{32},&&u_8 = v_{33}.& 
\end{align*}
Then, $\ku\{\mathcal{C}\} \simeq P_1$ as $\mathfrak{u}(\mgo)$-modules. Moreover, $V_0\simeq \ku\left\{v_{13}+v_{22}+v_{31}\right\}$ as $\mathfrak{u}(\mgo)$-modules and $V_1\otimes V_1\simeq V_0\oplus P_1$. 
\epf

\subsection{Tensoring by projective} We start by describing the decomposition of the tensor product between simple and projective modules.

\begin{prop}\label{simple-projective} 
$P_i\otimes V_1\simeq iP_0\oplus (3-i)P_1$,  $i \in \I_{0,1}$.
\end{prop}
\pf 
 By \eqref{ex_Omegar_1}, for $l=0$ and $r=1$, we have that the exact sequence
\begin{align*}
0 \to V_0 \to \Omega(V_0)  \to 2V_1 \to 0.
\end{align*}
Applying $-\otimes V_1$ to this sequence we obtain by the Proposition \ref{simple-simple} the following exact sequence
\begin{align*}
0 \to V_1 \to \Omega(V_0) \otimes V_1 \to 2V_0 \oplus 2P_1 \to 0.
\end{align*}
Then, there exists a $5$-dimensional $\mathfrak{u}(\mgo)$-module  $M$  such that  $ \Omega(V_0) \otimes V_1 \simeq M \oplus 2P_1$ and  $M$ fits in the exact sequence $0 \to V_1  \to M \to 2 V_0 \to 0$.
By definition, we have the following exact sequence $0 \to \Omega(V_0) \to P_0 \to V_0 \to 0$.
Tensoring the previous exact sequence by $V_1$, we obtain  the exact sequence
\begin{align*}
0 \to  M \oplus 2P_1 \to P_0\otimes V_1  \to  V_1 \to 0.
\end{align*}
Consequently, $P_0\otimes V_1  \simeq 2P_1 \oplus N$, for some projective module $N$ that fits in the exact sequence $0 \to M  \to N \to V_1 \to 0$.
Thus, $N \simeq P_1$ and hence $P_0 \otimes V_1 \simeq 3P_1$. In a similar way, $P_1 \otimes V_1 \simeq P_0 \oplus 2 P_1$.
\epf 

For each  $i \in \I_{0,1}$, denote by $[M:V_i]$ the number of composition factors isomorphic to $V_i$ of a module $M$. By induction on the length of $M$, we have that 
\begin{align}
\begin{aligned}\label{tensor_projective}
&M \otimes P_0 \simeq P_0 \otimes M \simeq [M:V_0]P_0 \oplus 3[M:V_1]P_1 ,\\[.2em]
& M \otimes P_1 \simeq P_1 \otimes M \simeq [M:V_1]P_0 \oplus (2[M:V_1]+[M:V_0])P_1.
\end{aligned}
\end{align}

\subsection{Tensoring syzygies by syzygies} In order to express the tensor pro\-duct between syzygies modules in a simpler way, from now on we consider the following notation:
\begin{align*}
&P(a,b)=aP_0\oplus bP_1, && a,b \in \N_0,& \\[.2em]
&\Omega^{s,t}_{i,j}=\Omega^s(V_i)  \otimes \Omega^t(V_j), && s,t\in \Z,\,\,i,j\in \I_{0,1}.
\end{align*}
Moreover, we establish the following convention $V_2=V_0$. The next result give us the decomposition of tensor product between syzygies modules.

\begin{prop}\label{syzygies_by_syzygies}
Let $s,t \in \N_0$ and $i,j \in \I_{0,1}$. Then, 
\[\Omega^{s,t}_{i,j}\simeq  \begin{cases}
\Omega^{s+t}(V_{i+j}) \oplus P(st,st),& \text{if }\, s+i,\,t+j\in\even,\\
\Omega^{s+t}(V_{i+j}) \oplus P(0,s(2t+1)),& \text{if }\, s+i \in \even,\,\, t+j \in \odd,\\
\Omega^{s+t}(V_{i+j}) \oplus P(0,(2s+1)t),& \text{if } \,s+i \in \odd,\,\,t+j \in \even, \\
\Omega^{s+t}(V_{i+j}) \oplus P(st, (s+1)(t+1)),& \text{if } \,s+i,\,t+j \in \odd.
\end{cases} \]
\end{prop}

\pf We prove the result for  $i=j=1$; the other cases have similar proofs. Assume that $t=0$. If $s=0$, then the result follows by Proposition \ref{simple-simple}. if $s=1$, then applying $-\otimes V_1$ in the exact sequence \eqref{sec-omega1-r-even}, we obtain by Propositions \ref{simple-simple} and \ref{simple-projective} that the sequence $0 \to \Omega^{1,0}_{1,1} \to P(1,2) \to V_0 \oplus P(0,1) \to 0$ is exact. By Corollary \ref{cor-epimorphism}, $\Omega^{1,0}_{1,1}\simeq \Omega(V_0)\oplus P(0,1)$.

So we can consider $s \geq 2$. Suppose that the isomorphism is valid for $s$ and we show for $s+1$. If $s$ even, then applying $-\otimes V_1$ in the sequence \eqref{sec-omega1-r-even} with $r=s$ and $l=1$, we obtain by Proposition \ref{simple-projective} and induction hypothesis the following exact sequence 

\[0 \to \Omega^{s+1,0}_{1,1} \to P(s+1,2(s+1)) \xrightarrow{f}  \Omega^{s}(V_0) \oplus P(0,s+1)\to 0\]
Note that $f=g+\id$, where $g:P(s+1,s+1)\to\Omega^{s}(V_0)$ is an epimorphism and $\id$ is the identity map on $(s+1) P_1$.
Applying Corollary \ref{cor-epimorphism} to the epimorphism $g$ we obtain that $\ker g=\Omega^{s+1}(V_0)\oplus (s+1)P_1$. Thus
\[\Omega^{s+1,0}_{1,1} \simeq \ker f\simeq \ker g\simeq \Omega^{s+1}(V_0) \oplus P(0,s+1).\]

Now suppose that $s$ odd. Applying $-\otimes V_1$ in the sequence \eqref{sec-omega1-r-odd} with $r=s$ and $l=1$, we obtain by Proposition \ref{simple-projective} and induction hypothesis the following exact sequence 
\begin{align*} 0 \to \Omega^{s+1,0}_{1,1} \to P(0,3(s+1)) \to \Omega^{s}(V_1) \oplus  P(0,s). \end{align*} 
Again by Corollary \ref{cor-epimorphism}, $\Omega^{s+1,0}_{1,1} \simeq \Omega^{s+1}(V_0) \oplus P(0,(s+2))$.

Now, we fix $s$ and show the isomorphism by induction over $t$. To do this, we proceed as follows. We apply  $\Omega^{s}(V_1) \otimes -$ to the exact sequences \eqref{sec-omega1-r-odd} or \eqref{sec-omega1-r-even} according to parity of $r=t$ and use Corollary \ref{cor-epimorphism}. We analyze four cases.

{\it Case 1:} $s,t$ even.

Applying  $\Omega^{s}(V_1) \otimes -$ in \eqref{sec-omega1-r-even} with $r=t$ and $l=1$, by \eqref{ex_Ujr}, \eqref{tensor_projective} and induction hypothesis we obtain 
\begin{align*} 0 &\to \Omega^{s,t+1}_{1,1}\to  P((s+1)(t+1), (3s+2)(t+1)) \to \\ &\to \Omega^{s+t}(V_0) \oplus P(st, (s+1)(t+1)) \to 0. \end{align*}
By Corollary \ref{cor-epimorphism}, $\Omega^{s,t+1}_{1,1} \simeq \Omega^{s+t+1}(V_0) \oplus P(0,2(s+1)(t+1)).$
\vspace{0,15cm}

{\it Case 2:} $s$ even, $t$ odd.
\begin{align*} 0 &\to \Omega^{s,t+1}_{1,1} \to  P(s(t+1), 3(s+1)(t+1)) \to  \Omega^{s+t}(V_0) \oplus P(0,(2s+1)t) \to 0. \end{align*}
 So, $\Omega^{s,t+1}_{1,1} \simeq \Omega^{s+t+1}(V_0) \oplus P(s(t+1), (s+1)(t+2)).$
\vspace{0,15cm}

{\it Case 3:} $s$ odd, $t$ even. 
\begin{align*} 0 &\to \Omega^{s,t+1}_{1,1} \to P(s(t+1), (3s+1)(t+1)) \to \Omega^{s+t}(V_0) \oplus P(0,s(2t+1)) \to 0. \end{align*}
Thus,  $\Omega^{s,t+1}_{1,1} \simeq \Omega^{s+t+1}(V_0) \oplus P(s(t+1), s(t+1)).$
\vspace{0,15cm}

{\it Case 4:} $s,t$ odd. 
\begin{align*} 0 &\to \Omega^{s,t+1}_{1,1} \to P((s+1)(t+1),3s(t+1)) \to \Omega^{s+t}(V_0) \oplus P(st, st) \to 0. \end{align*}
Therefore,  $\Omega^{s,t+1}_{1,1} \simeq \Omega^{s+t+1}(V_0) \oplus P(0,s[2(t+1)+1]).$
\epf

\subsection{Tensoring cosyzygies by cosyzygies} The tensor products of cosyzygies by cosyzygies are obtained dualizing the isomorphisms given in the Proposition \ref{syzygies_by_syzygies}. Precisely, we have the following.

\begin{coro}\label{tensor-cosy-by-cosy}
Let  $s,t \in \N_0$, $i,j \in \I_{0,1}$. Then, 
$$\Omega^{-s,-t}_{i,j} \simeq  \begin{cases}
\Omega^{-(s+t)}(V_{i+j}) \oplus P(st,st),& \text{if }\, s+i,\,t+j \in\even,\\
\Omega^{-(s+t)}(V_{i+j}) \oplus P(0,s(2t+1)),& \text{if }\, s+i \in \even,\,\, t+j \in \odd,\\
\Omega^{-(s+t)}(V_{i+j}) \oplus P(0,(2s+1)t),& \text{if }\, s+i \in \odd,\, t+j\in \even, \\
\Omega^{-(s+t)}(V_{i+j}) \oplus P(st,(s+1)(t+1)),& \text{if }\, s+i,\,t+j \in \odd.
\end{cases}$$
\end{coro}
\pf It follows from Proposition \ref{syzygies_by_syzygies}, Remark \ref{dualizing-projec-cover} and Remark \ref{dualizing_syzygies}.
\epf 

\subsection{Tensoring syzygies by cosyzygies}
Now we determine the decomposition of tensor product between syzygies and cosyzygies modules.

\begin{prop}\label{syzygies_by_cosyzygies-general}
Let $s,t \in \N_0$, $i,j \in \I_{0,1}$. Then, $\Omega^{s,-t}_{i,j}\simeq \Omega^{-t,s}_{j,i}$ and
$$\Omega^{s,-t}_{i,j} \simeq  \begin{cases}
\Omega^{s-t}(V_{i+j}) \oplus P(0,s(2t+1)),& \text{if }\, s+i \text{ even}\leq t+j \text{ even},\\
\Omega^{s-t}(V_{i+j}) \oplus P(s(t+1),s(t+1)),& \text{if }\, s+i \text{ even} < t+j \text{ odd},\\
\Omega^{s-t}(V_{i+j}) \oplus P(s(t+1),(s+1)t),& \text{if }\, s+i \text{ odd} < t+j \text{ even}, \\
\Omega^{s-t}(V_{i+j}) \oplus P(0,(2s+1)(t+1)) ,& \text{if }\, s+i \text{ odd} \leq t+j \text{ odd},\\
\Omega^{s-t}(V_{i+j}) \oplus P(0,(2s+1)t),& \text{if }\, s+i \text{ even} > t+j \text{ even},\\
\Omega^{s-t}(V_{i+j}) \oplus P((s+1)t,s(t+1)),& \text{if }\, s+i \text{ even} > t+j \text{ odd},\\
\Omega^{s-t}(V_{i+j}) \oplus P((s+1)t,(s+1)t),& \text{if }\, s+i \text{ odd} > t+j \text{ even}, \\
\Omega^{s-t}(V_{i+j}) \oplus P(0,(s+1)(2t+1)),& \text{if }\, s+i \text{ odd} > t+j \text{ odd}.
\end{cases}$$
\end{prop} 

\pf 
The proof follows a strategy analogous to that used for showing Proposition \ref{syzygies_by_syzygies}. We prove the result for $i=j=1$ and proceed by induction over $s,t$. Firstly, we suppose that $s\leq t$. Clearly, if $s=0$, then by Proposition  \ref{syzygies_by_syzygies}, Remark \ref{dualizing-projec-cover} and Remark \ref{dualizing_syzygies}, we have that 
\begin{align}\label{sygygie_cosyzygie_1}
& \Omega^{0,t}_{1,1} \simeq  \begin{cases}
\Omega^{-t}(V_{0}) \oplus P(0,t),& \text{if }\, t \in \odd,
\\ \Omega^{-t}(V_{0}) \oplus P(0,t+1),& \text{if }\, t \in \even.
\end{cases} 
\end{align}
If $s=1$, then applying $-\otimes \Omega^{-t}(V_1)$ in the sequence \eqref{sec-omega1-r-even}, with $r=0$, $l=1$ and $t$ even, we obtain by \eqref{tensor_projective} and \ref{sygygie_cosyzygie_1} that the sequence $$0 \to \Omega^{1,t}_{1,1} \to P(t+1,3t+2)  \to \Omega^{-t}(V_0) \oplus P(0,t+1) \to 0$$ is exact.
By Corollary \ref{cor-epimorphism}, $\Omega^{1,t}_{1,1}  \simeq \Omega^{1-t}(V_0) \oplus P(t+1,t+1).$ On the other hand, if $t$ is odd, then we have the following exact sequence $$0 \to \Omega^{1,t}_{1,1}  \to P(t,3t+1) \to \Omega^{-t}(V_0) \oplus P(0,t) \to 0.$$ Again by Corollary  \ref{cor-epimorphism}, $\Omega^{1,t}_{1,1}  \simeq \Omega^{1-t}(V_0) \oplus P(0,2t+1).$ Now, suppose that the result hold for $s$ and we show for $s+1$. Applying   $- \otimes \Omega^{-t}(V_1)$ in the exact sequences \eqref{sec-omega1-r-odd} or \eqref{sec-omega1-r-even} according to parity of $r=s$ and use Corollary \ref{cor-epimorphism}. We analyze four cases.

{\it Case 1:} $s,t$ even.

Applying  $- \otimes \Omega^{-t}(V_1)$ in \eqref{sec-omega1-r-even} with $r=s$ and $l=1$, by  \eqref{tensor_projective} and induction hypothesis we obtain 
\begin{align*} 0 &\to \Omega^{s+1,t}_{1,1}  \to  P((s+1)(t+1), (s+1)(3t+1)) \to \\ &\to \Omega^{s-t}(V_0) \oplus P(0, (2s+1)(t+1)) \to 0. \end{align*}
By Corollary \ref{cor-epimorphism}, $$\Omega^{s+1,t}_{1,1}  \simeq \Omega^{s+1-t}(V_0) \oplus P((s+1)(t+1), (s+1)(t+1)).$$

{\it Case 2:} $s$ even, $t$ odd.

Applying  $- \otimes \Omega^{-t}(V_1)$ in \eqref{sec-omega1-r-even} with $r=s$ and $l=1$, by  \eqref{tensor_projective} and induction hypothesis we obtain 
\begin{align*} 0 &\to \Omega^{s+1,t}_{1,1}  \to  P((s+1)t, (s+1)(3t+1)) \to \\ &\to \Omega^{s-t}(V_0) \oplus P(s(t+1), (s+1)t) \to 0. \end{align*}
By Corollary \ref{cor-epimorphism}, $$\Omega^{s+1,t}_{1,1}  \simeq \Omega^{s+1-t}(V_0) \oplus P(0, (s+1)(2t+1)).$$

{\it Case 3:} $s$ odd, $t$ even. 

Applying  $- \otimes \Omega^{-t}(V_1)$ in \eqref{sec-omega1-r-odd} with $r=s$ and $l=1$, by  \eqref{tensor_projective} and induction hypothesis we obtain 
\begin{align*} 0 &\to \Omega^{s+1,t}_{1,1}  \to  P((s+1)t, 3(s+1)(t+1)) \to \\ &\to \Omega^{s-t}(V_0) \oplus P((s+1)(t+1), (s+1)(t+1)) \to 0. \end{align*}
By Corollary \ref{cor-epimorphism}, $$\Omega^{s+1,t}_{1,1}  \simeq \Omega^{s+1-t}(V_0) \oplus P(0, [2(s+1)+1](t+1)).$$

{\it Case 4:} $s,t$ odd.

Applying  $- \otimes \Omega^{-t}(V_1)$ in \eqref{sec-omega1-r-odd} with $r=s$ and $l=1$, by  \eqref{tensor_projective} and induction hypothesis we obtain 
\begin{align*} 0 &\to \Omega^{s+1,t}_{1,1}  \to  P((s+1)(t+1), 3(s+1)t) \to \\ &\to \Omega^{s-t}(V_0) \oplus P(0,(s+1)(2t+1)) \to 0. \end{align*}
By Corollary \ref{cor-epimorphism}, $$\Omega^{s+1,t}_{1,1}  \simeq \Omega^{s+1-t}(V_0) \oplus P((s+1)(t+1), (s+2)t).$$
The other cases are similar.\epf

\subsection{Tensoring $(r,r)$-types by syzygies and by cosyzygies}
In order to present the decomposition of tensor product between $(r,r)$-types and syzygies (resp. cosyzygies), we need to prove an auxiliary result. In this subsection, $r\in \N$, $0\neq\xtt{x}\in \ku^2$, $\Att_{\xtt{x}}(r)$ and $\Btt_{\xtt{x}}(r)$ are the $(r,r)$-type modules defined in Subsection \ref{subsection-type}.
We recall that $\{z_i\,:\, i\in \I_{4r}\}$ denotes a basis of $\Att_{\xtt{x}}(r)$. Then, $\{u_{ij}:=z_i\otimes v_j\,:\,i\in \I_{4r},\,\,j\in \I_{3}\}$ is a basis of $\Att_{\xtt{x}}(r)\otimes V_1$. 

\begin{lemma} \label{lem-aux-type}
$\Att_{\xtt{x}}(r)\otimes V_1\simeq \Btt_{\xtt{x}}(r)\oplus rP_1$ and $\Btt_{\xtt{x}}(r)\otimes V_1\simeq \Att_{\xtt{x}}(r)\oplus rP_1$.
\end{lemma}
\begin{proof}
We proceed by induction on $r$. If $r=1$, then $\Att_{\xtt{x}}(1)\otimes V_1$ contains a submodule isomorphic to $\Btt_{\xtt{x}}(1)$. Indeed, it is easy to check that the vector space with basis $\{w,aw,a^2w,u_{23}+u_{32}+u_{41}\}$, where $w=u_{11}+x_2u_{33}$, is the desired submodule.
 Applying $-\otimes V_1$ in \eqref{1seq-Ax} for $r=1$, we obtain 
\[0 \to V_0\oplus P_1 \to \Att_{\xtt{x}}(1)\otimes V_1\to V_1 \to 0.\]
So $ \Att_{\xtt{x}}(1)\otimes V_1 \simeq N \oplus P_1$, for some submodule $N$ of $\Att_{\xtt{x}}(1)\otimes V_1$. Since $\Btt_{\xtt{x}}(1)$ is a submodule of $\Att_{\xtt{x}}(1)\otimes V_1$, $\soc \Btt_{\xtt{x}}(1)=V_0$ and $\soc P_1=V_1$ we have that $\Btt_{\xtt{x}}(1) \cap P_1 = \{0\}$. Moreover, $\dim (\Btt_{\xtt{x}}(1)\oplus P_1)=\dim (\Att_{\xtt{x}}(1)\otimes V_1)=12$ and consequently $\Att_{\xtt{x}}(1)\otimes V_1\simeq \Btt_{\xtt{x}}(1)\oplus P_1$. Dualizing, we obtain $\Btt_{\xtt{x}}(1)\otimes V_1\simeq \Att_{\xtt{x}}(1)\oplus P_1$.

Suppose $r>1$ and that the result holds to $r-1$. Applying $-\otimes V_1$ in \eqref{seq-Ax} for $r-1$ we get 
\[0 \to \Btt_{\xtt{x}}(1)\oplus P_1 \to \Att_{\xtt{x}}(r)\otimes V_1\to \Btt_{\xtt{x}}(r-1)\oplus (r-1)P_1 \to 0.\]
Thus, $\Att_{\xtt{x}}(r)\otimes V_1\simeq N\oplus rP_1$ where $N$ fitting in the exact sequence
\[0 \to \Btt_{\xtt{x}}(1)\to N\to \Btt_{\xtt{x}}(r-1) \to 0.\]
By Proposition \ref{exact-sequence-rrtypes}, $N\simeq \Btt_{\xtt{x}}(r)$ or $N\simeq \Btt_{\xtt{x}}(1)\oplus \Btt_{\xtt{x}}(r-1)$. Suppose the second possibility. Using the induction hypothesis and Proposition \ref{simple-projective} we have
\begin{align*}
(\Att_{\xtt{x}}(r)\otimes V_1)\otimes V_1&\simeq  (N\oplus rP_1)\otimes V_1 \simeq (\Btt_{\xtt{x}}(1)\oplus \Btt_{\xtt{x}}(r-1)\oplus rP_1 ) \otimes V_1 \\
&\simeq (\Btt_{\xtt{x}}(1)\otimes V_1)\oplus ( \Btt_{\xtt{x}}(r-1)\otimes V_1)\oplus ( rP_1 \otimes V_1)\\
&\simeq \Att_{\xtt{x}}(1)\oplus P_1\oplus  \Att_{\xtt{x}}(r-1)\oplus (r-1)P_1\oplus rP_0 \oplus2rP_1\\
&\simeq \Att_{\xtt{x}}(1)\oplus  \Att_{\xtt{x}}(r-1)\oplus rP_0 \oplus 3rP_1.
\end{align*}
On the other hand, using Proposition \ref{simple-simple}, it follows that
\begin{align*}
\Att_{\xtt{x}}(r)\otimes( V_1\otimes V_1)&\simeq   \Att_{\xtt{x}}(r)\otimes(V_0\oplus P_1)\simeq  \Att_{\xtt{x}}(r)\oplus (\Att_{\xtt{x}}(r)\otimes P_1)\\
&\simeq \Att_{\xtt{x}}(r)\oplus rP_0 \oplus 3rP_1.
\end{align*}
By Krull-Schmidt theorem we obtain a contradiction. Therefore, $N\simeq \Btt_{\xtt{x}}(r)$, consequently $\Att_{\xtt{x}}(r)\otimes V_1\simeq \Btt_{\xtt{x}}(r) \oplus rP_1$ and 
$\Btt_{\xtt{x}}(r)\otimes V_1\simeq(\Att_{\xtt{x}}(r)\otimes V_1)^{\ast}\simeq ( \Btt_{\xtt{x}}(r) \oplus rP_1)^\ast\simeq  \Att_{\xtt{x}}(r) \oplus rP_1$.
\end{proof}

\begin{prop} \label{prop-type-syzygye} Let $s\in \N_{0}$ and $i\in \I_{0,1}$. Then,
\begin{enumerate}[leftmargin=*,label=\rm{(\roman*)}]
\item $\Att_{\xtt{x}}(r)\otimes \Omega^s(V_i) \simeq \begin{cases}
\Att_{\xtt{x}}(r)\oplus P(rs,rs),& \text{if } s+i \in \even,\\
\Btt_{\xtt{x}}(r)\oplus P(0,r(2s+1)),& \text{if } s+i \in \odd,
\end{cases}$
\item $\Btt_{\xtt{x}}(r)\otimes\Omega^s(V_i)\simeq \begin{cases}
\Btt_{\xtt{x}}(r)\oplus P(0,2rs),& \text{if } s+i \in \even,\\
\Att_{\xtt{x}}(r)\oplus P(rs,r(s+1)),& \text{if } s+i \in  \odd.
\end{cases}$
\end{enumerate}
\end{prop}
\begin{proof}
We prove the item (i) for $i=1$. We proceed by induction over $s$. If $s=0$, then the result follows from Lemma \ref{lem-aux-type}. Suppose that $s=1$. Applying $ \Att_{\xtt{x}}(r) \otimes -$ in \eqref{sec-omega1-r-even} with $r=0$ and $l=1$ we obtain by \eqref{tensor_projective} and Lemma \ref{lem-aux-type} that $ 0 \to \Att_{\xtt{x}}(r) \otimes \Omega(V_1) \to P(r,3r) \to \Btt_{\xtt{x}}(r) \oplus P(0,r) \to 0$. By Corollary \ref{cor-epimorphism} and Proposition \ref{seq1-Ax} $ \Att_{\xtt{x}}(r) \otimes \Omega(V_1) \simeq  \Att_{\xtt{x}}(r) \oplus P(r,r)$. Suppose that the result holds for $s\geq 2$. Now, applying $\Att_{\xtt{x}}(r) \otimes -$ in \eqref{sec-omega1-r-even} or in \eqref{sec-omega1-r-odd} according to parity of $s$ we obtain the result. The item (ii) is similar.
\end{proof}

\begin{coro} \label{cor-type-cosyzygye}
 Let $s\in \N_{0}$ and $i\in \I_{0,1}$. Then, 
\begin{enumerate}[leftmargin=*,label=\rm{(\roman*)}]
\item $\Att_{\xtt{x}}(r)\otimes\Omega^{-s}(V_i) \simeq \begin{cases}
\Att_{\xtt{x}}(r)\oplus P(0,r(2s+1)),& \text{if } s+i \in \even,\\
\Btt_{\xtt{x}}(r)\oplus P(rs,r(s+1)),& \text{if } s+i \in \odd,
\end{cases}$ 
\item $\Btt_{\xtt{x}}(r)\otimes \Omega^{-s}(V_i) \simeq \begin{cases}
\Btt_{\xtt{x}}(r)\oplus P(rs,rs),& \text{if } s+i \in \even,\\
\Att_{\xtt{x}}(r)\oplus P(0,r(2s+1)),& \text{if } s+i \in \odd.
\end{cases} $
\end{enumerate} 
\end{coro}
\begin{proof}
The result follows directly from Proposition \ref{prop-type-syzygye} by dualization.
\end{proof}

\subsection{Tensoring $(r,r)$-types by $(r,r)$-types}
Throughout this subsection, $r,s \in \N$, $0\neq \xtt{x}, \xtt{y}\in \ku^2$, $\Att_{\xtt{x}}(r)$ and $\Btt_{\xtt{y}}(s)$ are the modules defined in $\S\,$\ref{subsection-type}. \vspace{.1cm}

Let $\{z_i\,:\, i\in \I_{4r}\}$ and $\{z'_j\,:\, j\in \I_{4s}\}$ basis of $\Att_{\xtt{x}}(r)$ and $\Att_{\xtt{y}}(s)$, respectively. Hence,
$\{u_{ij}:=z_i\otimes z'_j\,:\,i\in \I_{4r},\,\,j\in \I_{4s}\}$ is a basis of $\Att_{\xtt{x}}(r)\otimes \Att_{\xtt{y}}(s)$. Also, if $\{z''_j\,:\, j\in \I_{4s}\}$ is a basis of $\Btt_{\xtt{y}}(s)$, then  $\{\tilde{u}_{ij}:=z_i\otimes z''_j\, : \, i\in \I_{4r},\, j\in \I_{4s} \}$ is a basis of $\Att_{\xtt{x}}(r)\otimes \Btt_{\xtt{y}}(s)$.

\begin{lemma} \label{lem:aux-type1}
Let $\xtt{x},\xtt{y}\in\Bbbk^2$ with $\overline{\xtt{x}}\neq\overline{\xtt{y}}$. Then,
\begin{align*}
\Att_{\xtt{x}}(1)\otimes\Att_{\xtt{y}}(1)\simeq P(1,1), & & \Att_{\xtt{x}}(1)\otimes\Btt_{\xtt{y}}(1)\simeq P(0,2), & & \Btt_{\xtt{x}}(1)\otimes\Btt_{\xtt{y}}(1)\simeq P(1,1).
\end{align*}
\end{lemma}

\begin{proof}
Suppose that $\xtt{x}=(x_1,x_2)$, $\xtt{y}=(y_1,y_2)$ and $\overline{\xtt{x}}\neq\overline{\xtt{y}}$. Let $w, w^{\prime} \in \Att_{\xtt{x}}(1)\otimes\Att_{\xtt{y}}(1)$ given by \begin{align*}
w&=y_2u_{12}+x_2u_{21}+x_2y_2(u_{34}+u_{43})\\[.1em]
w^{\prime}&=u_{11}+(x_2y_1)u_{33}+(x_1y_2+y_1x_2)u_{42}.
\end{align*}
Observe that $\{w,aw,a^2w,a^3w,w^{\prime},aw^{\prime},a^2w^{\prime},a^3w^{\prime}\}$ is a basis of $P_0$. On the other hand, setting $r=1$ in \eqref{1seq-Ax} and tensorizing by $\Att_{\xtt{y}}(1)$ we obtain from Proposition \ref{prop-type-syzygye} that $P_1$ is a submodule of $\Att_{\xtt{x}}(1)\otimes\Att_{\xtt{y}}(1)$. Since $\soc P_0=V_0$ and $\soc P_1=V_1$, we obtain that $\Att_{\xtt{x}}(1)\otimes\Att_{\xtt{y}}(1)\simeq P(1,1)$ by dimension argument.

Now, notice that we have two copies of $P_1$ into $\Att_{\xtt{x}}(1)\otimes\Btt_{\xtt{y}}(1)$. In fact, for $i\in \I_2$, consider 
$\beta_i=\{w_i,aw_i,a^2w_i,a^3w_i,w_i^{\prime},aw_i^{\prime},a^2w_i^{\prime},a^3w_i^{\prime}\}$ where
\begin{align*}
&w_1=\tilde{u}_{21}+y_2\tilde{u}_{34},& &w_1^{\prime}=\tilde{u}_{22}+y_2\tilde{u}_{44},& \\
&w_2=y_2\tilde{u}_{14}+x_2(\tilde{u}_{23}+\tilde{u}_{32}+\tilde{u}_{41}),&&w_2^{\prime}=\tilde{u}_{11}+x_2\tilde{u}_{33}.& 
\end{align*}
It is clear that $\beta_i$ is a basis of $P_1$, $i \in \I_2$. Using that $\overline{\xtt{x}}\neq \overline{\xtt{y}}$, it is straightforward to verify that $\beta_1\cup \beta_2$ is linearly independent. From dimension argument, we obtain $\Att_{\xtt{x}}(1)\otimes\Btt_{\xtt{y}}(1)\simeq P(0,2)$.
Finally, dualizing the isomorphism $\Att_{\xtt{x}}(1)\otimes\Att_{\xtt{y}}(1)\simeq  P(1,1)$ we obtain  $\Btt_{\xtt{x}}(1)\otimes\Btt_{\xtt{y}}(1)\simeq  P(1,1)$.
\end{proof}

For the next result we fix the following notation. Let $j,k\in \mathbb{Z}$. If $j<k\in \N_0$, then $\binom{j}{k}$ denote the usual binomial coefficient. Moreover, we convention that $\binom{j}{k}=0$, when  $j< k$ or $j\neq k$ and $k<0$; and  $\binom{k}{k}=1$, for all $k\in \mathbb{Z}$.

\begin{lemma}\label{lem:aux-type2}
Let $r,s \in \N$, $0\neq \xtt{x}\in \ku^2$ and $l \in \I_{0,1}$. Then,
\begin{enumerate}[leftmargin=*,label=\rm{(\roman*)}]
\item  $0 \to \Att_{\xtt{x}}(r) \to \Omega^{r}(V_l) \to V_0 \to 0$ is exact, when $r+l \in \even$;
\item   $0 \to \Btt_{\xtt{x}}(r) \to \Omega^{r}(V_l) \to V_1 \to 0$ is exact, when $r+l \in \odd$;
\item  $\Att_{\xtt{x}}(r)\otimes \Btt_{\xtt{x}}(s)$ has a submodule isomorphic to $\Btt_{\xtt{x}}(t)$, where $t=\min\{r,s\}$.
\end{enumerate}
\end{lemma}
\begin{proof} 
Suppose that $r+l$ is even. Then, by Theorem \ref{teo-syzygy}, $\Omega^{r}(V_l) \simeq \Utt_{1,r}$. Let $\{z_i: i \in \I_{4r}\}$ be a basis of $\Att_{\xtt{x}}(r)$ and $\{w_i: i \in \I_{4r+1}\}$ be a basis of $\Utt_{1,r}$. We proceed by cases.

\noindent{\it Case $\xtt{x}=(x_1,0)$}. Define the linear map $\iota:\Att_{\xtt{x}}(r)\to \Utt_{1,r}$ by 
\[\iota(z_{4j+1})=x_1\sum_{t=0}^{j}w_{4t+1},\qquad \iota(z_{4j+i})=\sum_{t=0}^{j}w_{4t+i},\] where $j\in \I_{0,r-1}$, $i\in \I_{2,4}$.
It is clear that $\iota$ is an injective module morphism. Now, let $V_{0}= \ku\{v_0\}$ and consider the linear map $\pi:\Utt_{1,r}\to V_0$ defined by $\pi(w_{4r+1})=v_0$,  $\pi(w_k)=0, \,\,k\in \I_{4r}$.
Clearly $\pi$ is a surjective module morphism such that $\ker \pi=\operatorname{Im}\iota$. \vspace{.1cm}

\noindent{\it Case $\xtt{x}=(0,x_2)$}.  Define the linear map $\iota:\Att_{\xtt{x}}(r)\to \Utt_{1,r}$ by
\[\iota(z_{4j+1})=x_2\sum_{t=0}^{j}w_{4(r-t)+1},\qquad \iota(z_{4j+i})=\sum_{t=0}^{j}w_{4(r-t-1)+i}, \]
where $j\in \I_{0,r-1}$, $i\in \I_{2,4}$.
Clearly $\iota$ is an injective module morphism and the linear map $\pi:\Utt_{1,r}\to V_0$ defined by $\pi(w_{1})=v_0$, $ \pi(w_k)=0$, $k\neq 1$, is a surjective module morphism such that $\ker \pi=\operatorname{Im}\iota$.\vspace{.1cm} 
 
\noindent{\it Case $\xtt{x}=(x,x)$}.   
Consider a family of nonzero scalars $\{\lambda_i\}_{i\in\I_r}\subseteq \Bbbk$ such that $\lambda_i\neq\lambda_j, i\neq j$. For any $j,k\in \I_{0,r}$, we define
\[c_{j,k}=\binom{j}{j-k}\lambda_1+\binom{j}{j-k+1}\lambda_2 + \cdots+ \binom{j}{j-1}\lambda_k+\binom{j}{j}\lambda_{k+1}\in \Bbbk,\]
and the linear map $\iota:\Att_{\xtt{x}}(r)\to \Utt_{1,r}$  by: 
\[\iota(z_{4l+1})=\sum_{j=0}^{r}c_{j,l}w_{4j+1},\qquad \iota(z_{4l+i})=\sum_{j=0}^{r-1}c_{j,l}w_{4j+i}, \]
$l\in \I_{0,r-1}$, $i\in \I_{2,4}$.
It is straightforward to verify that $\iota$ is an injective module morphism and $\pi:\Utt_{1,r}\to V_0$ defined by
$\pi(w_{1})=\pi(w_{4r+1})=v_0$, $\pi(w_l)=0$, $l \in \I_{2,4r}$, is a surjective module morphism such that $\ker \pi=\operatorname{Im}\iota$.\vspace{.1cm}

\noindent{\it Case $\xtt{x}=(x_1,x_2)$, $x_1,x_2\neq 0$}. Given $j,k\in \I_{0,r}$, we consider
\begin{align*}
c_{j,k}&=\sum_{l=0}^{k}\binom{j-1+l}{j-1}\binom{j}{k-l}x_1^{r-j+(k-2l)}x_2^{j-(k-2l)},\\
\tilde{c}_{j,k}&=\sum_{l=0}^{k}\binom{j+l}{j}\binom{j}{j-l}x_1^{r-j-1+(k-2l)}x_2^{j-(k-2l)}.
\end{align*}
The linear map $\iota:\Att_{\xtt{x}}(r)\to \Utt_{1,r}$ defined  by: 
\[\iota(z_{4l+1})=\sum_{j=0}^{r}c_{j,l}w_{4j+1},\qquad \iota(z_{4l+i})=\sum_{j=0}^{r-1}\tilde{c}_{j,l}w_{4j+i},  \]
$l\in \I_{0,r-1}$, $i\in \I_{2,4}$, is an injective module morphism. Moreover,  $\pi:\Utt_{1,r}\to V_0$ defined by
$\pi(w_{1})=x_2^rv_0$, $\pi(w_{4r+1})=x_1^rv_0$, $\pi(w_k)=0$, $k \in \I_{2,4r}$, is a surjective module morphism such that $\ker \pi=\operatorname{Im}\iota$.

\vspace{.1cm}

The proof of (ii) can also be carried out by explicitly exhibiting the morphisms that ensure the exactness of the sequence, and will be omitted. For (iii), let $\{z_i: i \in \I_{4r}\}$ and $\{z'_j: j \in \I_{4s}\}$ basis of $\Att_{\xtt{x}}(r)$ and $\Btt_{\xtt{x}}(s)$, respectively. Thus, $\{u_{ij}=z_i\otimes z'_j:  i \in \I_{4r}, j \in \I_{4s}\}$ is a basis of $\Att_{\xtt{x}}(r) \otimes \Btt_{\xtt{x}}(s)$. For each $k\in \I_{0,t-1}$, consider the vectors $w_{4k+4} = \sum_{i+j=4k+5} u_{ij}$ and
\[w_{4k +3}(x_1,x_2) = x_2\Bigg(\sum_{\substack{ i+j=4k\\i \equiv  3\!\!\!\!\mod 4}} u_{ij}\Bigg)  + \Bigg( \sum_{\substack{i+j=4(k+1)\\i\equiv  1 \!\!\!\! \mod 4}} \big(u_{ij} + x_1 u_{ji}\big)\Bigg).\]
Now, consider the set $\beta(x_1,x_2)=\cup_{k \in \I_{0,t-1}}\beta_k(x_1,x_2)$, where
\begin{align*}
\beta_k(x_1,x_2)&=\{b^2w_{4k+3}(x_1,x_2),bw_{4k+3}(x_1,x_2),w_{4k+3}(x_1,x_2),w_{4k+4}\}.
\end{align*}
Assume that $0\neq \xtt{x}=(x_1,x_2)$ with $x_1\neq x_2$. In this case, it is straightforward to check that $\beta(x_1,x_2)$ is an independent linear set and the vector space generated by $\beta(x_1,x_2)$ is a module isomorphic to  $\Btt_{\xtt{x}}(t)$. If $0\neq \xtt{x}=(x,x)$, then $\beta(1,0)$ is an independent linear set and the vector space generated by $\beta(1,0)$ is a module isomorphic to  $\Btt_{\xtt{x}}(t)$.
\end{proof}

\begin{prop} \label{tensor-type-type}
Let $r,s \in \N$, $\xtt{x},\xtt{y}\in\Bbbk^2$ and $t=\min\{r,s\}$. Then,
\begin{enumerate}[leftmargin=*,label=\rm{(\roman*)}]
\item $\Att_{\xtt{x}}(r)\otimes\Att_{\xtt{y}}(s) \simeq \begin{cases}
P(rs,rs),&  \text{if $\,\,\overline{\xtt{x}}\neq\overline{\xtt{y}}$},\\
\Att_{\xtt{x}}(t)\oplus \Btt_{\xtt{x}}(t)\oplus P(rs-t,rs) ,& \text{if $\,\,\overline{\xtt{x}}=\overline{\xtt{y}}$},
\end{cases} $\\[.5em] 
\item $\Att_{\xtt{x}}(r)\otimes\Btt_{\xtt{y}}(s) \simeq \begin{cases}
2rsP_1,&  \text{if $\,\,\overline{\xtt{x}}\neq\overline{\xtt{y}}$},\\
\Att_{\xtt{x}}(t)\oplus \Btt_{\xtt{x}}(t)\oplus P(0,2rs-t) ,& \text{if $\,\,\overline{\xtt{x}}=\overline{\xtt{y}}$},
\end{cases}$\\[.5em] 
\item $\Btt_{\xtt{x}}(r)\otimes\Btt_{\xtt{y}}(s) \simeq \begin{cases}
P(rs,rs),&  \text{if $\,\,\overline{\xtt{x}}\neq\overline{\xtt{y}}$},\\
\Att_{\xtt{x}}(t)\oplus \Btt_{\xtt{x}}(t)\oplus P(rs-t,rs) ,& \text{if $\,\,\overline{\xtt{x}}=\overline{\xtt{y}}$}.
\end{cases}$
\end{enumerate} 
\end{prop}

\begin{proof}
First, we analyze the case $\overline{\xtt{x}} \neq\overline{\xtt{y}}$.

\noindent (i) If $r=s=1$, then the result holds by Lemma \ref{lem:aux-type1}. Assume that $s=1$ and that the result holds for $r$. Applying $-\otimes A_{\xtt{y}}(1)$ in the exact sequence \eqref{seq-Ax}, we obtain 
$$0 \to P(1,1) \rightarrow \Att_{\xtt{x}}(r+1)\otimes \Att_{\xtt{y}}(1)  \rightarrow P(r,r)\to 0.$$
Hence, $\Att_{\xtt{x}}(r+1)\otimes \Att_{\xtt{y}}(1)\simeq P(r+1,r+1)$ and the result follows. Finally, assuming that $\Att_{\xtt{x}}(r)\otimes\Att_{\xtt{y}}(s)\simeq P(rs,rs)$ and applying $A_{\xtt{x}}(r) \otimes  - $ in the exact sequence \eqref{seq-Ax} it follows that 
$$0 \to P(s,s) \rightarrow \Att_{\xtt{x}}(r)\otimes \Att_{\xtt{y}}(s+1)  \rightarrow P(rs,rs)\to 0.$$
So, $\Att_{\xtt{x}}(r+1)\otimes \Att_{\xtt{y}}(s)\simeq P(r(s+1),r(s+1))$.

\noindent (ii) Applying $\Att_{\xtt{x}}(r)\otimes -$ in the exact sequence \eqref{seq1-Bx}, we obtain by (i) and \eqref{tensor_projective} that the sequence $0 \to \Att_{\xtt{x}}(r)\otimes \Btt_{\xtt{y}}(s) \to P(rs,3rs) \to P(rs,rs) \to 0$ is exact. Consequently, the result follows.

\noindent(iii) Follows by dualization of (i).

Now, suppose that $\overline{\xtt{x}}=\overline{\xtt{y}}$. So,  $\overline{\xtt{x}}=(x_1,x_2)$ and $\overline{\xtt{y}}=\lambda(x_1,x_2)$, $\lambda\in \Bbbk^{\times}$. 

\noindent(ii) Since $\xtt{x} \sim \xtt{y}$, then $ \Btt_{\xtt{x}}(s) \simeq  \Btt_{\xtt{y}}(s)$. Suppose that $s \leq r$. By Lemma \ref{lem:aux-type2} (i), the sequence $0 \to \Att_{\xtt{x}}(r) \to \Omega^{r}(V_l) \to V_0 \to 0$ is exact, when $r+l$ is even. Applying $- \otimes \Btt_{\xtt{x}}(s)$, by Proposition \ref{prop-type-syzygye} (ii), we obtain that \[
\xymatrix{
0 \ar[r] & \Att_{\xtt{x}}(r) \otimes \Btt_{\xtt{x}}(s) \ar[r]^-{\iota} &  \Btt_{\xtt{x}}(s) \oplus P(0,2rs) \ar[r]^-{\pi} &  \Btt_{\xtt{x}}(s) \ar[r] & 0} \]
is exact. By Lemma \ref{lem:aux-type2} (iii), $\Att_{\xtt{x}}(r) \otimes \Btt_{\xtt{x}}(s)$ has a submodule $N$ isomorphic to $\Btt_{\xtt{x}}(s)$. Since $\iota$ is a monomorphism, $\iota(N) \simeq N \simeq \Btt_{\xtt{x}}(s)$. Moreover, $\iota(N) \cap 2rsP_1 = \{0\}$, because $\soc(\Btt_{\xtt{x}}(s))=sV_0$ and $\soc(2rsP_1)= 2rsV_1$. Consequently, the sequence 
\[\xymatrix{
0 \ar[r] & \Att_{\xtt{x}}(r) \otimes \Btt_{\xtt{x}}(s) \ar[r]^-{\iota} &  \iota(N) \oplus P(0,2rs) \ar[r]^-{\pi} &  \Btt_{\xtt{x}}(s) \ar[r] & 0} \] is exact. As $\pi \circ \iota=0$, then the restriction $\pi\mid_{2rsP_1}: 2rsP_1 \to \Btt_{\xtt{x}}(s)$ is an epimorphism. By Corollary \ref{cor-epimorphism} and \eqref{seq1-Ax}, we conclude that $\ker(\pi\mid_{2rsP_1}) \simeq \Att_{\xtt{x}}(s) \oplus (2r-1)s P_1$. Therefore, $\Att_{\xtt{x}}(r)\otimes\Att_{\xtt{y}}(s) \simeq \Att_{\xtt{x}}(s)\oplus\Btt_{\xtt{x}}(s) \oplus P(0, (2r-1)s).$ The proof for the case $s>r$ is analogous, however we start with the sequence given in the Lemma \ref{lem:aux-type2} (ii).

\noindent(iii) Observe that $(V_1 \otimes \Att_{\xtt{x}}(r)) \otimes \Btt_{\xtt{x}}(s) \simeq \Btt_{\xtt{x}}(r) \otimes \Btt_{\xtt{x}}(s) \oplus P(rs,3rs)$. On the other hand, $V_1 \otimes (\Att_{\xtt{x}}(r) \otimes \Btt_{\xtt{x}}(s)) \simeq \Btt_{\xtt{x}}(t) \oplus \Att_{\xtt{x}}(t) \oplus P(2rs-t, 4rs)$. By Krull-Schmidt theorem, $ \Btt_{\xtt{x}}(r) \otimes \Btt_{\xtt{x}}(s) \simeq \Att_{\xtt{x}}(t)\oplus \Btt_{\xtt{x}}(t)\oplus P(rs-t,rs)$.

\noindent(i) Follows by dualization of (iii).
\end{proof}

\section{Green ring}
Let $H$ be a Hopf algebra. The {\it Green ring} of $H$, denoted by $r(H)$, is the abelian group generated by the isomorphism classes of finite dimensional $H$-modules subject to the relations $[V\oplus W]=[V]+[W]$, where $[V]$ and $[W]$ denote the isomorphism classes of the finite dimensional $H$-modules $V$ e $W$ respectively. The multiplication in $r(H)$ is induced by the tensor product between $H$-modules, i.~e., $[V][W]=[V\otimes W]$. If $H$ is a quasitriangular Hopf algebra then $V\otimes W \simeq W\otimes V$ as $H$-modules, for any $H$-modules $V$ and $W$. In this case, $r(H)$ is a commutative ring.

By \cite[Theorem 2.9]{ABDF}, $\mathfrak{u}(\mgo)$ is homomorphic image of the Drinfeld double $D(H)$ of a finite dimensional Hopf algebra $H$. Since $D(H)$ is quasitriangular, $\mathfrak{u}(\mgo)$ is also quasitriangular. Hence $r(\mathfrak{u}(\mgo))$ is a commutative ring. We recall, for what follows, that the  indecomposable $(s,s)$-type modules $\Att_{\xtt{x}}(s)$ and  $\Btt_{\xtt{x}}(s)$ were defined in $\S\,$\ref{subsection-type}, for all $s\in \N$ and $\xtt{x}\in \ku^2$. In order to present $r(\mathfrak{u}(\mgo))$ via generators and relations, we fix:
\begin{align*} 
& 1=[V_0], & & \ogr=[V_1], & & \tgr=[P_1],\\& \thgr_3=[\Omega(V_0)],&  & \tilde{\thgr}_3=[\Omega^{-1}(V_0)], & &  \fgr_{\xtt{x},s}=[\Btt_{\xtt{x}}(s)],
\end{align*} 
$s \in \N, \xtt{x}\in \ku^2$. Explicitly, we show that the set $$\{1,\ogr,\tgr, \ogr\tgr,\thgr_3^n, \ogr\thgr_3^n,  \tilde{\thgr}_3^n,  \ogr\tilde{\thgr}_3^n,  \fgr_{\xtt{x},s}^n, \ogr \fgr_{\xtt{x},s}^n  \,: \,n,s \in \N,\, \xtt{x}\in \ku^2 \}$$ is a $\mathbb{Z}$-basis of $r(\mathfrak{u}(\mgo))$.

\begin{lemma}\label{relation-green-0}
The following relations holds in $r(\mathfrak{u}(\mgo))$.
\begin{align}
\label{gr-rel-1} &\ogr^2=1+\tgr, \\
\label{generation-1} & [P_0]=\tgr \ogr-2\tgr,\\
\label{gr-rel-2} &\tgr^2=2(\ogr+1)\tgr, \\
\label{gr-rel-3} &\thgr_3\tgr=\tgr(2\ogr+1), \\
\label{gr-rel-4} &\tilde{\thgr}_3\tgr=\tgr(\ogr+1), \\
\label{gr-rel-5} &\thgr_3\tilde{\thgr}_3=1+6\tgr.
\end{align} 
\end{lemma}
\begin{proof}
The relations \eqref{gr-rel-1} and \eqref{gr-rel-5} follows from Propositions \ref{simple-simple} and  \ref{syzygies_by_cosyzygies-general}, respectively. By Proposition \ref{simple-projective} we obtain \eqref{generation-1}. The other relations follows from \eqref{tensor_projective} and \eqref{generation-1}.
\end{proof}

\begin{lemma}\label{relation-green}
Let $n \in \N$. Then, $[\Omega^n(V_0)]= {\thgr}_3^n-f_n\tgr$, where

\[f_n= \begin{cases}
0,& \text{if }\, n=1,\\
f_{n-1}(2\ogr+1)+(n-1)\ogr+2,& \text{if }\, n \in \even,\\
f_{n-1}(2\ogr+1)+3(n-1),& \text{if } \,1 \neq n \in \odd.
\end{cases}\]
\end{lemma}

\begin{proof}
We proceed by induction on $n$. Since $\thgr_3=[\Omega(V_0)]$, the result follows for $n=1$. By Proposition \ref{syzygies_by_syzygies}, $[\Omega^2(V_0)] = {\thgr_3}^2-(\ogr+2)\tgr$ and we have the result. Now, suppose that the result holds for $n$. If $n+1$ is even, then by Proposition \ref{syzygies_by_syzygies}, \eqref{generation-1} and induction hypothesis
\begin{align*}
[\Omega^{n+1}(V_0)]&= [\Omega^n(V_0)] [\Omega(V_0)] -  n[P_0] - 2(n+1)[P_1]\\&=( {\thgr}_3^n - f_n \tgr ) {\thgr}_3 - n \ogr \tgr + 2 n \tgr- 2n \tgr - 2 \tgr \\& = {\thgr}_3^{n+1} -f_n(2\ogr+1)\tgr-(n \ogr +2) \tgr\\&= {\thgr}_3^{n+1} -(f_n(2\ogr+1)+n \ogr +2) \tgr\\&= {\thgr}_3^{n+1} - f_{n+1}  \tgr.
\end{align*}

The case $n+1$ odd is similar.
\end{proof}


\begin{coro}\label{relation-green1}
Let $n \in \N$. Then, $[\Omega^{-n}(V_0)]= \tilde{\thgr}_3^{n}-f_n\tgr$.
\end{coro}

\pf Since $1=[V_0]=[V_0^\ast]$ and $\ogr=[V_1]=[V_1^\ast]$ the result follows from de previous lemma.
\epf

Let $R$ be the subring of $r(\mathfrak{u}(\mgo))$ generated by $\ogr, \tgr, \thgr_3 \mbox{ and } \tilde{\thgr}_3$. 

\begin{prop}\label{z-Basis}
The set $\{1,\ogr,\tgr, \ogr\tgr,\thgr_3^n, \ogr\thgr_3^n,  \tilde{\thgr}_3^n,  \ogr\tilde{\thgr}_3^n: n \in \N \}$ is a $\mathbb{Z}$-basis of $R$.
\end{prop}
\begin{proof}
Let $R_1$ be the $\mathbb{Z}$-submodule of  $r(\mathfrak{u}(\mgo))$ generated by $[V_i]$, $[P_i]$, $[\Omega^n(V_i)]$ and $[\Omega^{-n}(V_i)]$, $i \in \I_{0,1}, n \in \N$. Clearly, $R_1$ is a free $\mathbb{Z}$-submodule with this basis. By the previous results of this section, we have that $R_1$ is a subring of $r(\mathfrak{u}(\mgo))$ such that $R \subseteq R_1$. Conversely, by \eqref{generation-1}, Lemma \ref{relation-green} and Corollary \ref{relation-green1}, we have that $[P_0]$, $[\Omega^n(V_0)] \in R$, $n \in \mathbb{Z}$. By Proposition \ref{syzygies_by_syzygies} and Corollary \ref{tensor-cosy-by-cosy}, we have that
\begin{align}\label{z-basis}
\begin{aligned} 
[\Omega^n(V_1)]=\begin{cases}
 \ogr \thgr_3^n-( \ogr f_n+n)\tgr,& \text{if }\, n \in \even,\\
 \ogr \thgr_3^n-( \ogr f_n+n+1)\tgr,& \text{if } \, n \in \odd,
\end{cases}\\
[\Omega^{-n}(V_1)]=\begin{cases}
\ogr  \tilde{\thgr}_3-( \ogr f_n+n)\tgr,& \text{if }\, n \in \even,\\
\ogr  \tilde{\thgr}_3-( \ogr f_n+n+1)\tgr,& \text{if } \, n \in \odd.
\end{cases}
\end{aligned} 
\end{align} Consequently, $[\Omega^n(V_1)] \in R$, $n \in \mathbb{Z}$. Therefore, $R=R_1$.

 Now, consider $R_2$ the subring of $r(\mathfrak{u}(\mgo))$ generated by $\ogr$ and  $\tgr$. So, by Lemma \ref{relation-green-0}, $R_2$ is a free $\mathbb{Z}$-submodule with $\mathbb{Z}$-basis $\{1, \ogr, \tgr, \ogr \tgr\}$. Take $\pi: R \to R/R_2$ the canonical projection of $\mathbb{Z}$-modules. Since $f_n \tgr, \ogr f_n \tgr \in R_2$, follows by \eqref{z-basis} that $\pi([\Omega^n(V_i)])= \ogr^i\thgr_3^n$ and $\pi([\Omega^{-n}(V_i)])= \ogr^i\tilde{\thgr_3}^n$, $i \in \I_{0,1}, n \in \N$. Therefore, since $\{[V_i], [P_i], [\Omega^n(V_i)], [\Omega^{-n}(V_i)]: i \in \I_{0,1}, n \in \N\}$ is a $\mathbb{Z}$-basis of $R$, the result hold. 
\end{proof}

Let $\mathbb{Z}[x_1,x_2,x_3,x_4]$ be the polynomial algebra over $\mathbb{Z}$ in the commutative variables $x_1,x_2,x_3$ and $x_4$. Consider $I^{\prime}$ the ideal of $\mathbb{Z}[x_1,x_2,x_3,x_4]$ generated by the elements: 
\begin{align}\begin{aligned} \label{polynomial}
&x_1^2-x_2-1,&& x_2^2-2x_1x_2-2x_2, && x_2x_3-2x_1x_2-x_2, \\& x_2x_4-x_1x_2-x_2, & & x_3x_4-6x_2-1.\end{aligned} 
\end{align}

Therefore, we have the following. 

\begin{theorem}
$R \simeq \mathbb{Z}[x_1,x_2,x_3,x_4]/I^{\prime}$ as rings. \qed
\end{theorem}

\begin{lemma}
Let $s,t\in \N$, $\xtt{x}, \xtt{y} \in  \ku^2$, $\overline{\xtt{x}}\neq\overline{\xtt{y}}$. We have the following relations: 
\begin{align}
\label{generation-4} &[\Att_{\xtt{x}}(s)]=\fgr_{\xtt{x},s}\ogr-s\tgr,\\
\label{gr-rel-6}& \fgr_{\xtt{x},s}\tgr=s\tgr(\ogr+1), \\
\label{gr-rel-7} & \fgr_{\xtt{x},s}\thgr_3= \fgr_{\xtt{x},s}\ogr+s\tgr(\ogr-1), \\
\label{gr-rel-8} &\fgr_{\xtt{x},s}\tilde{\thgr}_3=\fgr_{\xtt{x},s}\ogr+2s\tgr,\\
\label{gr-rel-9} &\fgr_{\xtt{x},s}\fgr_{\xtt{y},t}=st \tgr(\ogr-1),\\
\label{gr-rel-10} &\fgr_{\xtt{x},s}\fgr_{\xtt{x},t}=\fgr_{\xtt{x},s}\ogr+\fgr_{\xtt{x},s}+(st-s)\tgr(\ogr-1),\quad s\leq t.
\end{align}
\end{lemma}
\begin{proof}
By Proposition \ref{prop-type-syzygye}, we have \eqref{generation-4}. The equation \eqref{gr-rel-6} follows from \eqref{tensor_projective} and \eqref{generation-1}. Moreover,  \eqref{gr-rel-7} follows from Proposition \ref{prop-type-syzygye}, \eqref{generation-1} and \eqref{generation-4}. For \eqref{gr-rel-8} we use the Corollary \ref{cor-type-cosyzygye} and \eqref{generation-4}. Finally, by Proposition \ref{tensor-type-type}, \eqref{generation-1} and \eqref{generation-4} we obtain \eqref{gr-rel-9} and \eqref{gr-rel-10}.
\end{proof}
 
Let $\mathbb{Z}[X]$ be the polynomial algebra over $\mathbb{Z}$ in the commutative variables $X=\{x_i, Z_{\xtt{x},s}: i \in \I_{1,4}, \,\xtt{x}\in \mathbb{P}_1(\Bbbk), s \in \mathbb{N} \}$. Let $s,t\in \N$, $\xtt{x}, \xtt{y} \in  \ku^2$, $\overline{\xtt{x}}\neq\overline{\xtt{y}}$ and $I$ the ideal of $\mathbb{Z}[X]$ generated by the relations 
\begin{align}\begin{aligned}\label{polynomial2}
& Z_{\xtt{x},s}x_2-sx_2(x_1+1),\\& Z_{\xtt{x},s} (x_3 -x_1)-s x_2(x_1-1),\\
 &Z_{\xtt{x},s}(x_4-x_1) -2s x_2,\\&Z_{\xtt{x},s}Z_{\xtt{y},t} - st x_2(x_1-1),\\
 &Z_{\xtt{x},s}(Z_{\xtt{x},t} -x_1-1)  - s(t-1)x_2(x_1-1), \, s\leq t, \end{aligned}
 \end{align} and \eqref{polynomial}. Now, we have the main result of this work.

\begin{theorem} \label{green-ring}
 $r(\mathfrak{u}(\mgo)) \simeq \mathbb{Z}[X]/I$ as rings.
\end{theorem}
\begin{proof}
There exists a unique ring homomorphism  $\varphi: \mathbb{Z}[X] \to r(\mathfrak{u}(\mgo))$ such that $\varphi(x_1)= \ogr$, $\varphi(x_2) = \tgr$, $\varphi(x_3) = \thgr_3$, $\varphi(x_4)= \tilde{\thgr}_3$, $\varphi(Z_{\xtt{x},s})=\fgr_{\xtt{x},s}$, $\xtt{x}\in \mathbb{P}_1(\Bbbk), s \in \mathbb{N}$. By \eqref{generation-4} and the proof of Proposition \ref{z-Basis}, $\{\ogr, \tgr,  \thgr_3, \tilde{\thgr}_3, \fgr_{\xtt{x},s}: \xtt{x}\in \mathbb{P}_1(\Bbbk), s \in \mathbb{N} \}$ generated $r(\mathfrak{u}(\mgo))$ as  ring. So, $\varphi$ is an epimorphism. By \eqref{gr-rel-1}, \eqref{gr-rel-2} - \eqref{gr-rel-5}, \eqref{gr-rel-6} - \eqref{gr-rel-10} we have that $\varphi(I)=0$. Consequently, $\varphi$ induce a ring epimorphism $\overline{\varphi}: \mathbb{Z}[X]/I \to r(\mathfrak{u}(\mgo))$. Using the same arguments from the proof of \cite[Theorem 3.9]{chen}, we conclude that $\overline{\varphi}$ is a ring isomorphism.
\end{proof}

\begin{remark} 
We would like to point out a correction that need to be made in \cite{vay}.
It is written in Equation (4.12) pg. 40 that for odd $s$, it holds that 
$M_k(\varepsilon,\mathbf{t})\otimes \Omega^s(\varepsilon)\simeq M_k(\varepsilon,\mathbf{t})\oplus skP_{\varepsilon}\oplus 2(s+1)kP_L.$
The correct formula is
\begin{align}\label{formula-vay-corrigida-2}
M_k(\varepsilon,\mathbf{t})\otimes \Omega^s(\varepsilon)\simeq M_k(L,\mathbf{t})\oplus skP_{\varepsilon}\oplus 2(s+1)kP_L.
\end{align}
Using the correct formula \eqref{formula-vay-corrigida-2}, the relation (Rel 7) in Table 2 of \cite[p. 25]{vay} changes from 
$\mu_{k,\mathbf{t}}\omega=k\lambda\rho+\mu_{k,\mathbf{t}}$ to $\mu_{k,\mathbf{t}}\omega=-2k\rho+k\lambda\rho+\lambda\mu_{k,\mathbf{t}}$.
\end{remark}

\section{Semisimplification of $\rep \mathfrak{u}(\mgo)$}

We recall that $\mathfrak{u}(\mgo)$ is a Hopf algebra and $\Ss(a)=a$, $\Ss(b)=b$ and $\Ss(c)=c$, where $\Ss$ denotes the antipode map. Hence, 
$\Ss^2(a^ib^jc^k)=a^ib^jc^k$, for all $i,j\in \I_{0,3}$ and $k\in \I_{0,1}$. Hence, it follows from Remark \ref{basis-um} that $\mathfrak{u}(\mgo)$ is an involutory Hopf algebra. Thus, by \cite[Example 3.2]{bw}, $\mathfrak{u}(\mgo)$ is spherical with spherical element $\omega=1$.

Let us revisit the notion of semisimplification $\underline{\rep}\,\mathfrak{u}(\mgo)$. Firstly, we remember that the quantum dimension $\qdim V$ of $V\in \rep \mathfrak{u}(\mgo)$ is the trace of the action of $\omega=1$ on $V$, i.~e., $\qdim V=(\dim V)1_{\Bbbk}$. It is well known that $\underline{\rep}\,\mathfrak{u}(\mgo)$ is a semisimple category. Moreover, there is a bijective correspondence between the non-isomorphic simple objects in $\underline{\rep}\,\mathfrak{u}(\mgo)$ and the the set of isomorphism classes of indecomposable finite-dimensional $\mathfrak{u}(\mgo)$-modules with non-zero quantum dimension. Also, the quantum dimension of every projective module is zero.
We refer to \cite{bw} and \cite{egno} for more details on semisimplification category.

By Subsections \ref{subsec-string-modules} and \ref{subsec-band-modules}, the $(r,r)$-type modules $\Vtt_{i,t},\Wtt_{i,t}, \Att_{\lambda,n}$ and  $\Btt_{\lambda,n}$ have 
even dimension and consequently they have quantum dimension zero. Since $P_0$ and $P_1$ also have quantum dimension zero, the simple objects in $\underline{\rep}\,\mathfrak{u}(\mgo)$ are the syzygy modules $\Omega^s(V_i)$ and the cosyzygy modules $\Omega^{-s}(V_i)$, $s\in \N_0$ and $i\in \I_{0,1}$. Recall that, by convention, $\Omega^0(V_i)=V_i$.

Consider the group $\Gamma=C_2\times \Z$, where $C_2$ is the cyclic group of order $2$.
Denote by $\operatorname{vect}_{\Bbbk}^{\Gamma}$ the category of $\Gamma$-graded finite-dimensional $\Bbbk$-vector spaces and suppose that $x$ and $z$ are the generators of the cyclic groups $C_2$ and $\Z$ respectively.

\begin{theorem}
The functor $F:\underline{\rep}\,\mathfrak{u}(\mgo)\to \operatorname{vect}_{\Bbbk}^{\Gamma}$ defined by 
\[F(\Omega^s(V_i))=\Bbbk_{x^iz^s},\qquad i\in \I_{0,1},\, s\in \Z,\]
determines a monoidal equivalence between  $\underline{\rep}\,\mathfrak{u}(\mgo)$ and $\operatorname{vect}_{\Bbbk}^{\Gamma}$.
\end{theorem}
\begin{proof}
It follows from Theorem \ref{green-ring}.
\end{proof}

\noindent \emph{Acknowledgements.} We would like to thank E. R. Alvares  for his  comments on Lemma \ref{lemma-chen-geral}
and O. Marquez for his suggestions on Lemma \ref{lem:aux-type2}. We also thanks  C. Vay for fruitful discussions.


\begin{thebibliography}{99}


\bibitem{ABDF} N. Andruskiewitsch, D. Bagio, S. Della Flora and D. Flores. \emph{On the Drinfeld double of the restricted Jordan plane in characteristic 2}. 
J. Pure Appl. Algebra  \textbf{229} (1), p. 107798 (2025).

\bibitem{As} I. Assem. \emph{Algèbres et modules}. Enseignement des mathématiques, Presses Université Ottawa (1997).

\bibitem{ars} M. Auslander, I. Reiten and S. Smal{\o}. \emph{Representation theory of Artin algebras}. 
Camb. Stud. Adv. Math. \textbf{36}, Cambridge University Press  (1995).


\bibitem{bw} J. W. Barrett and B. W. Westbury. \emph{Spherical Categories}.  Adv. Math. \textbf{143} (2), 357-375  (1999).

\bibitem{carlson} J. F. Carlson, L. Townsley, L. Valero-Elizondo and M. Zhang. \emph{Cohomology Rings of Finite Groups}. 
Springer-Verlag, Berlin, Heidelberg (2003).

\bibitem{chen1} H.-X. Chen. \emph{Finite-dimensional representation of a quantum double}. J. Algebra {\bf 251}, 751-789 (2002).


\bibitem{chen} H.-X. Chen.. \emph{The Green ring of Drinfeld double $D(H_4)$}. Algebr. Represent. Theory {\bf 17}, 1457-1483 (2014).

\bibitem{cushingetal} D. Cushing, G. W. Stagg and D. I. Stewart. \emph{A Prolog assisted search for new simple Lie algebras}. Math. Comput. {\bf 93}, 1473-1495 (2022).

\bibitem{do} A. Kh. Dolotkazin. \emph{Irreducible representations of a three-dimensional simple Lie algebra of characteristic p=2}. 
Math. Notes \textbf{24}, 588-590 (1979).

\bibitem{egno} P. Etingof, S. Gelaki, D. Nikshych and V. Ostrik. \emph{Tensor Categories}. Mathematical Surveys and Monographs, Providence, RI (2015).


\bibitem{G} J. A. Green. \emph{The modular representation algebra of a finite group}. Ill. J. Math. \textbf{6} (4), 607-619 (1962).

\bibitem{HL} F. Huard and S. Liu. \emph{Tilted Special Biserial Algebras}. J. Algebra \textbf{217}, 679-700 (1999).


\bibitem{KS} H. Kondo and Y. Saito. \emph{Indecomposable decomposition of tensor products of modules over
the restricted quantum universal enveloping algebra associated to $\mathfrak{sl}_2$}. J. Algebra  \textbf{330}, 103-129 (2011).

\bibitem{LZ} L. Li and Y. Zhang. \emph{The Green rings of the generalized Taft Hopf algebras}. in Hopf algebras
and tensor categories, Contemp. Math. \textbf{585}, American Mathematical Society, Providence,
RI, 275-288 (2013).


\bibitem{Ra} D. E. Radford.  \emph{Hopf algebras}. Series on Knots and Everything 49. Hackensack, NJ: World Scientific. xxii, 559 p. (2012).

\bibitem{str} H. Strade. \emph{Simple Lie algebras over fields
of positive characteristic}. de Gruyter Expo. Math. \textbf{38}, (2004).

\bibitem{str1} H. Strade. \emph{Lie algebras of small dimension}.  Lie algebras, Vertex Operator Algebras and Their Applications, Contemp. Math. \textbf{442}, 233-265 (2007).

\bibitem{vay} C. Vay. \emph{The Green ring of a family of copointed Hopf algebras}. Rev. Un. Mat. Argentina. \textbf{68} (1), (2025).


\end{thebibliography}
\end{document}